\documentclass[11pt]{article}
\usepackage[english]{babel}
\usepackage{amssymb, amsmath, amsthm,mathrsfs} 
\usepackage[a4paper, centering]{geometry}
\usepackage{bbm,amsmath,amsthm,epsfig,latexsym,marvosym, esint}
\usepackage{amsfonts}
\usepackage{bigints}
\usepackage{amsmath}
\usepackage{bbm}
\usepackage{amssymb}
\usepackage{enumerate}
\usepackage{ esint }
\usepackage[pdftex,breaklinks,colorlinks, linkcolor=blue]{hyperref}
\usepackage{color}
\usepackage[font={small,up}]{caption}
\newcommand{\res}{\mathop{\hbox{\vrule height 7pt width .5pt depth 0pt
\vrule height .5pt width 6pt depth 0pt}}\nolimits}

\numberwithin{equation}{section}

\usepackage{amsthm}
\makeatletter
\def\th@plain{%
  \thm@notefont{}
  \itshape 
}
\def\th@definition{%
  \thm@notefont{}
  \normalfont 
}
\makeatother
\newtheorem{thm}{Theorem}[section]
\newtheorem{prop}[thm]{Proposition}
\newtheorem{lem}[thm]{Lemma}

\newtheorem{remark}[thm]{Remark}

\theoremstyle{definition}
\newtheorem{defn}[thm]{Definition}

\newtheorem{oss}[thm]{Remark}

\newcommand{\Z}{\mathbb{Z}}

\newcommand{\N}{\mathbb{N}}
\newcommand{\R}{\mathbb{R}}

\newcommand{\B}{\mathcal{B}}
\newcommand{\A}{\mathcal{A}}
\newcommand{\mS}{\mathbb{S}}
\renewcommand{\epsilon}{\varepsilon}

\newcommand{\ssubset}{\subset \! \subset}
\renewcommand{\vec}{\mathbf}    
\newcommand{\vecg}{\boldsymbol}
\newcommand{\Rnn}{\mathbb{R}^{n \times n}}
\newcommand{\vn}{\vec{n}}
\newcommand{\dd}{\mathrm{d}}
\newcommand{\La}{\mathcal{L}}
\newcommand{\rhotilde}{\tilde{\vecg\rho}}
\newcommand{\weakcs}{\overset{*}{\rightharpoonup}}

\newcommand{\weakc}{\rightharpoonup}
\newcommand{\weakly}{\rightharpoonup}
\newcommand{\hvpe}{\widetilde{\vecg\varphi}_\eta}

\DeclareMathOperator{\Det}{Det}

\DeclareMathOperator{\dist}{dist}

\DeclareMathOperator{\cof}{cof}
\DeclareMathOperator{\adj}{adj}

\DeclareMathOperator{\imT}{im_{T}}
\DeclareMathOperator{\imG}{im_{G}}

\DeclareMathOperator{\mec}{mec}
\DeclareMathOperator{\nem}{nem}

\DeclareMathOperator{\ess}{ess}
\DeclareMathOperator{\osc}{osc}

\def\e{\epsilon}
\usepackage{enumitem}
\setlist[description]{nosep}

\makeindex

\title{Relaxation of nonlinear elastic energies related to Orlicz-Sobolev nematic elastomers}

\author{Giovanni Scilla and Bianca Stroffolini }
\newcommand{\Addresses}{{
  \bigskip
  \footnotesize

G.~Scilla, \textsc{Dipartimento di Matematica ed Applicazioni ``R. Caccioppoli'', Universit\`{a} di Napoli Federico II, Via Cintia Monte Sant'Angelo, 80126 Napoli, Italy}\par\nopagebreak
  \textit{E-mail address}, G.~Scilla: \texttt{giovanni.scilla@unina.it}

\medskip

B.~Stroffolini, \textsc{Dipartimento di Ingegneria Elettrica e delle Tecnologie dell'Informazione, Universit\`{a} di Napoli Federico II, Via Claudio, 80125 Napoli, Italy}\par\nopagebreak
  \textit{E-mail address}, B.~Stroffolini: \texttt{bstroffo@unina.it}
}}
\date{}
\begin{document}

\maketitle

\begin{abstract} 
We compute the relaxation of the total energy related to a variational model for nematic elastomers, involving a nonlinear elastic mechanical energy depending on the orientation of the molecules of the nematic elastomer, and a nematic Oseen--Frank energy in the deformed configuration. The main assumptions are that the quasiconvexification of the mechanical term is polyconvex and that the deformation belongs to an Orlicz-Sobolev space with an integrability just above the space dimension minus one, and does not present cavitation. We benefit from the fine properties of orientation-preserving maps satisfying that regularity requirement proven in \cite{HS} and extend the result of \cite{MCOl} to Orlicz spaces with a suitable growth condition at infinity.
\end{abstract}

\tableofcontents

\section{Introduction}

The longstanding problem in nonlinear elasticity about the formation of cavities has been proposed in the keystone paper of Ball \& Murat~\cite{BaMu84}.
Subsequently, M\"uller \& Spector~\cite{MuSp95} investigated with many examples what are the conditions to impose in order to prevent cavities.
They introduced a topological condition: \lq INV \rq , formulated in terms of degree for maps in $W^{1,p}(\Omega, \mathbb{R}^n), p>n$. Later, Conti \& De Lellis~\cite{CoDeL} were relaxing this condition to map in $W^{1,n-1}$ obtaining some partial results.
Henao \& Mora Corral extensively studied Lusin properties and local invertibility,  \cite{HM10, HM11, HM12, HM15}.
Later on, Barchiesi, Henao \& Mora-Corral~\cite{HBMC}, using these  latter results, were able to  characterize the class of orientation preserving  maps in $W^{1,p}, p>n-1$ for which the cavities do not occur. 
This conditions are related to distributional Jacobian,  surface measure and degree. 

 In a previous paper \cite{HS} the second author with Henao proved that many properties of orientation preserving maps, such as local invertibility and a.e. differentiability, can be pushed to a special class of Orlicz-Sobolev spaces, with an integrability exponent just above the space dimension minus one, in the logarithmic scale. 

This kind of generalization could have not only a mathematical interest per se but it is also related to questions of integrability of Jacobian determinants and mappings of finite distortion {(see, e.g., \cite{HM, HK, IM,V,GISS})}.

The drawback of this generalization is an existence theorem for models of magnetic elastomers, liquid crystals and magnetoelasticity, see, e.g., \cite{BDesimone, calderer, kruzik}.  The  existence theorems were proved in the scale of Sobolev spaces with $p>n-1$  in \cite{HBMC}.
Both the theorems were provided assuming polyconvexity in the mechanic energy and quadratic growth in the deformed configuration (nematic). \par
The scope of this  article is the study of the magnetoelastic model without any policonvexity or quasiconvexity assumption. This gives rise to finding the quasiconvex envelope.
However, it can be shown that the two energies could be treated separately:
the quasiconvex envelope is the sum of the two envelopes: the quasiconvex for the mechanical, the tangential quasiconvexification for the nematic term (see \cite{MCOl} for the case $t^p, p>n-1$).\\
\\
\noindent
{\bf Main results:} Let $A(t)$ be a fixed $N$-function such that $A(t)\sim t^{n-1}\log^\alpha (e+t)$ for some ${\alpha\in(n-2,n-1)}$, and denote by $W^{1,A}(\Omega,\R^n)$ the corresponding Orlicz-Sobolev space. The energy $I(\vec u, \vn)$ associated to the deformation $\vec u$ and the director $\vn$ is the sum of two terms: (a) the mechanical energy of the deformation $I_{\mec}(\vec u, \vn)$, of the form
\begin{equation*}
\displaystyle \int_{\Omega} W(D\vec u(\vec x),\vn(\vec u(\vec x)))\, \dd \vec x,
\end{equation*} 
where the potential $W$ is assumed to comply with an Orlicz-growth condition (with respect to the deformation) as 
\begin{equation}
|W(\vec F, \vec n)| \leq c_W \left(A(\| \vec F \|) + h(\det \vec F) + 1 \right), \qquad \vec F \in \Rnn_+ , \quad \vn \in \mS^{n-1}
\label{condWintro} 
\end{equation} 
for a suitable convex function $h$, and an equicontinuity property with respect to $\vn$. The domain of $I_{\mec}$ is the class $\mathcal{A}_{{\beta}}(\Omega)$ of the admissible deformations, consisting of those maps in $W^{1,A}(\Omega,\R^n)$ which are orientation preserving and such that no cavitation occurs (see Definition~\ref{admissible}) { with a supplementary integrability condition for the cofactor (see Section~\ref{sec:cut})}; (b) the nematic energy in the deformed configuration $I_{\nem}(\vec u, \vn)$, given by
\begin{equation*}
\displaystyle \int_{\imT(\vec u,\Omega)} V(\vn(\vec y), D\vn(\vec y)) \, \dd \vec y,
\end{equation*}
where $\imT(\vec u,\Omega)$ denotes the \emph{topological image} of $\Omega$ by $\vec u$ (see Definition~\ref{de:imTomega}) and $|V(\vec z, \vecg\xi)|\leq C(1+A(\|\vecg\xi\|))$ for every $\vec z\in\mS^{n-1}$, $\vecg\xi \in (T_{\vec z} \mS^{n-1})^n$, $T_{\vec z} \mS^{n-1}$ being the tangent space at $\vec z$ to $\mS^{n-1}$.

Under the additional assumption that $W^{qc}$ - the quasiconvexification of $W$ in the first variable - is polyconvex, with Theorem~\ref{relaxation} we prove that the relaxation on $L^1 (\Omega, \R^n) \times L^1 (\R^n, \R^n)$ of the functional
\begin{equation}
\begin{split}
I(\vec u, \vn):&= I_{\mec}(\vec u, \vn) + I_{\nem}(\vec u, \vn)\\
                   &=\displaystyle \int_{\Omega} W(D\vec u(\vec x),\vn(\vec u(\vec x)))\, \dd \vec x +  \displaystyle \int_{\imT(\vec u,\Omega)} V(\vn(\vec y), D\vn(\vec y)) \, \dd \vec y
\end{split}
\end{equation}
is the energy given by the sum of the relaxed energies $I^*_{\mec}$ and $I^*_{\nem}$; namely,
\begin{equation}
\begin{split}
I^*(\vec u, \vn):&= I^*_{\mec}(\vec u, \vn) + I^*_{\nem}(\vec u, \vn)\\
                   &=\displaystyle \int_{\Omega} W^{qc}(D\vec u(\vec x),\vn(\vec u(\vec x)))\, \dd \vec x +  \displaystyle \int_{\imT(\vec u,\Omega)} V^{tqc}(\vn(\vec y), D\vn(\vec y)) \, \dd \vec y,
\end{split}
\end{equation} 
where $V^{tqc}$ is the tangential quasiconvexification of $V$ (see Definition~\ref{defn:tangq}).

The lower bound is an immediate consequence of the lower semicontinuity result for $I$ with respect to the strong topology of $L^1$ provided by Proposition~\ref{prop:lowersemicont}, which in turn relies on the lower semicontinuity properties of $I_{\mec}$ and $I_{\nem}$. As for the upper bound, with Proposition~\ref{prop:limsup} we exhibit the construction of a mutual recovery sequence $\{(\vec u_j, \vn_j)\}$ providing a limsup inequality for both the mechanical term and the nematic term separately. Our argument provides an extension to the logarithmic scale of Orlicz-Sobolev spaces of the approach by Conti-Dolzmann \cite{CoDo15} and Mora Corral-Oliva \cite{MCOl}.\\
\\
\noindent
{\bf Overview of the paper:} The paper is organized as follows. In Section~\ref{notation} we fix the main notation which will be used throughout the paper. Section~\ref{orlicz} collects some basic definitions and results in Convex Analysis, concerning $N$-functions and the Orlicz-Sobolev spaces. Then, with Section~\ref{prelim}, we recall basic definitions in geometric measure theory, necessary in order to tackle the analysis of the energy functionals in the deformed configuration. In particular, we recall the notions of geometric image (Definition~\ref{de:O0}),  the concept of topological degree in Orlicz-Sobolev maps and define the topological image of a set (Section~\ref{sec:degtop}), \cite{HS}. The class of admissible deformations $\mathcal{A}(\Omega)$ is introduced in Section~\ref{sec:admissible}, where we recall also their fine properties (Section~\ref{sec:fine}); in particular, openness and local invertibility, and we investigate their stability under composition with Lipschitz functions, useful for the change of variables (Section~\ref{sec:cut}). {Here we introduce the subclass  $\mathcal{A}_{\beta}(\Omega)$.} In Section~\ref{complow} we state the main results of compactness, lower semicontinuity (Proposition~\ref{prop:lowersemicont}) and existence of minimizers (Theorem~\ref{th:existence}) for the functionals defined in the deformed configuration. Finally, in order to obtain the main relaxation theorem (Theorem~\ref{relaxation}), we provide the construction of a recovery sequence (Theorem~\ref{recseq}) in the spirit of Conti-Dolzmann's approach (Lemma~\ref{approxubyz} and Proposition~\ref{prop:limsup}). There, an extension to the Orlicz-Sobolev setting of the concepts of tangential quasiconvexity and the corresponding results of lower semicontinuity and relaxation {come} into play (Theorem~\ref{th:tqc}).

\section{Notation}\label{notation}

In this section we fix the notation and introduce some definitions used in the paper.

Let $n\geq3$. In all the paper, $\Omega$ will be a non-empty open, bounded set of $\R^n$, which represents the body in its reference configuration. There, the coordinates will be denoted by $\vec x$, while in the deformed configuration by $\vec y$. Vector-valued and matrix-valued functions will be written in boldface. The closure of a set $A$ is denoted by $\bar{A}$ and its topological boundary by $\partial A$. Given two sets $U,V$ of $\R^n$, we will write $U \ssubset V$ if $U$ is bounded and $\bar{U} \subset V$. The open ball of radius $r>0$ centred at $\vec x \in \R^n$ is denoted by $B(\vec x, r)$, while $\overline{B(\vec x, r)}$ stands for its closure; when $\vec x=\vec 0$, we will simply write $B_r$ and $\overline{B_r}$, respectively. The $(n-1)$-dimensional sphere in $\R^n$ centred at $\vec x_0$, with radius $r$, is denoted by $S(\vec x_0; r)$ or $S_r(\vec x_0)$.
Given a square matrix $\vec M \in \Rnn$, its determinant is denoted by $\det \vec M$.
The adjugate matrix $\adj \vec M \in \Rnn$ satisfies $(\det \vec M) \vec I = \vec M \adj \vec M$, where $\vec I$ denotes the identity matrix.
The transpose of $\adj \vec M$ is the cofactor $\cof \vec M$.
If $\vec M$ is invertible, its inverse is denoted by $\vec M^{-1}$.
The inner product of vectors and of matrices will be denoted by $\cdot$ and their associated norms are denoted by $\left\| \cdot \right\|$.
Given $\vec a, \vec b \in \R^n$, the tensor product $\vec a \otimes \vec b$ is the $n \times n$ matrix whose component $(i,j)$ is $a_i \, b_j$.
The set $\Rnn_+$ denotes the subset of matrices in $\Rnn$ with positive determinant. 
The set $\mS^{n-1}$ denotes the unit sphere in $\R^n$. The identity function in $\R^n$ is denoted by $\vec {id}$.

The Lebesgue measure in $\R^n$ is denoted by $\left| \cdot \right|$ or $\La^n$, and the $(n-1)$-dimensional Hausdorff measure by $\mathcal{H}^{n-1}$. The abbreviation \emph{a.e.} stands for \emph{almost everywhere} or \emph{almost every}; unless otherwise stated, it refers to $\La^n$.
For $\Phi$ a Young function,  $L^\Phi$ denotes the corresponding Orlicz space and $W^{1,\Phi}, W^{1,\Phi}_0$ the Orlicz-Sobolev spaces  (see Section~\ref{orlicz} for the precise definitions).
The symbols $C^1_c$ and $C^\infty_c$ stand for the spaces of $C^1$ and $C^\infty$ functions, respectively, with compact support.
The derivative of a Sobolev-Orlicz or a smooth vector-valued function $\vec u$ is written $D \vec u$. The set of (positive or vector-valued) Radon measures is denoted by $\mathcal{M}$, while $BV$ is the space of functions with bounded variation.

The strong convergence in $L^\Phi$ or $W^{1,\Phi}$ and the a.e. convergence are denoted by $\to$, while the symbol for the weak convergence is $\weakc$, that for the weak$^*$ convergence in $L^{\infty}$ is $\weakcs$. Given a measurable set $A$, the symbol $\fint_A \vec u(\vec x)\,\dd\vec x$ denotes the average value of $\vec u$ on $A$.

\section{Some basic facts on Orlicz-Sobolev spaces}\label{orlicz}

We recall here few basic definitions and results concerning $N$-functions and Orlicz-Sobolev spaces, useful in the sequel. For a detailed treatment of the topic, we refer to \cite{Kras, Kufn, Bennett, Adams}. \\

An $N$-function $A$ is a convex function from $[0,\infty)$ to $[0,\infty)$ which vanishes only at 0 and such that 
$$\lim_{s\to 0^+}\frac{A(s)}{s}=0 \,\quad, \, \quad  \lim_{s\to\infty} \frac{A(s)}{s}=\infty.$$
 If $A$ is an $N$-function, then we denote by $A^*$ the \emph{Young-Fenchel-Yosida} dual or conjugate transform of $A$; namely, the $N$-function defined as
\begin{equation}
A^*(s):=\sup\{sr-A(r):\,0<r<+\infty\}.
\end{equation}

In this paper, we restrict our analysis to functions $A$ whose growth at infinity is at least such that
 \begin{equation} \label{eq:L2logL}
 \int_{t_0}^\infty \left ( \frac{t}{A(t)}\right)^\frac{1}{n-2}\mathrm{d}t < \infty.
 \end{equation}
for some $t_0\geq 0.$
The condition is satisfied, in particular, when $A(t)=t^{n-1}\log^\alpha (e+t)$ for every $\alpha>n-2$.

An $N$-function $A$ is said to satisfy the \emph{$\Delta_2$-condition near infinity} if it is finite-valued and there exist a constant $\mu>2$ and $t_0>0$ such that
 \begin{equation}
  A(2t)\leq \mu A(t)\quad \text{for}\ t\geq t_0.
\label{delta2}
 \end{equation}
If \eqref{delta2} holds for every $t>0$, we say that $A$ satisfies the \emph{$\Delta_2$-condition globally}.

\begin{remark}\label{delta22}
 We notice that our function $A(t)=t^{n-1}\log^\alpha (e+t)$ for every $\alpha>n-2$ verifies the $\Delta_2$ condition together with its conjugate. {We will also been dealing with the function $B(t)=t\log^\beta (e+t)$ for a $\beta>0$ (see Section~\ref{sec:admissible}):} this function verifies the $\Delta_2$-condition globally. It is worth noting that its conjugate $B^*$ is equivalent to $C(t)=e^{t^{1/\beta}}-1$ (see the remarks below Theorem~\ref{inclusion}) that, instead, does not satisfy $\Delta_2$-condition at infinity, since it holds that $C(2t)>2C(t)$ definitely for $t>0$. 
 \end{remark}
 An equivalent property to \eqref{delta2}, very useful in the computations, is the following: for every constant $\gamma>1$, there exists a constant $c_{\mu,\gamma}>0$ such that
\begin{equation}
A(\gamma t) \leq c_{\mu,\gamma} A(t)\quad \text{for}\ t\geq t_0.
\end{equation}

Let $\Omega$ be a measurable subset of $\R^n$. The Orlicz space $L^A(\Omega)$ built upon a Young function $A$ is the Banach function space of those real-valued measurable functions $u$ on $\Omega$ for which the Luxemburg norm
\begin{equation}
\|u\|_{L^A(\Omega)}:=\inf\left\{\lambda>0:\, \int_\Omega A\left(\frac{|u(\vec x)|}{\lambda}\right)\,\mathrm{d}\vec x\leq1\right\}
\end{equation}
is finite. 

Since $A$ is non-decreasing,
 \begin{align}
  \int_\Omega A(|u(\vec x)|)\dd\vec x <\infty \ \Rightarrow\ \|u\|_{L^A(\Omega)}\leq 1.
 \end{align}
If $A$ satisfies the $\Delta_2$-condition at infinity then 
 \begin{align} \label{eq:slicing1}
   u\in L^A(\Omega)\ \Leftrightarrow\ \int_\Omega A(|u(\vec x)|)\dd\vec x <\infty.
 \end{align}
 
 \begin{prop}[generalized H\"older inequality] Let $A$ be an $N$-function and $A^*$ its dual. Then it holds that
\begin{equation}
\left|\int_\Omega u(\vec x)v(\vec x)\,\mathrm{d}\vec x\right|\leq 2 \|u\|_{L^A(\Omega)}\|v\|_{L^{A^*}(\Omega)},
\end{equation}
for every $u\in L^A(\Omega)$ and $v\in L^{A^*}(\Omega)$.
\end{prop}

Note that we may introduce another norm on $L^A(\Omega)$, the \emph{Orlicz norm} or \emph{dual norm}, defined as
\begin{equation}
|u|_A:=\sup\left\{\int_\Omega u(\vec x)v(\vec x)\,\mathrm{d}\vec x:\, v\in L^{A^*}(\Omega), \|v\|_{L^{A^*}(\Omega)}\leq1\right\}.
\end{equation}
The norms $\|\cdot\|_{L^A(\Omega)}$ and $|\cdot|_A$ are equivalent, since it holds that
\begin{equation}
\|u\|_{L^A(\Omega)}\leq|u|_A\leq 2\|u\|_{L^A(\Omega)}, \quad u\in L^A(\Omega).
\end{equation}

We denote by $E^A(\Omega)$ the closure of all bounded measurable functions defined on $\Omega$ with respect to the norm $\|\cdot\|_{L^A(\Omega)}$. Now we remark that the $\Delta_2$-condition comes into play for separability and reflexivity:
\begin{prop}  $A$ satisfies $\Delta_2$-condition iff $E^A(\Omega)=L^A(\Omega)$.
\end{prop}
 If $A$ does not satisfy $\Delta_2$-condition, it turns out that
\begin{equation*}
E^A(\Omega)\subsetneq L^A(\Omega),
\end{equation*}
Moreover, $E^A(\Omega)$ is separable and $C^\infty_0(\Omega)$ is dense in $E^A(\Omega)$.\\

The Orlicz space $L^A(\Omega,\R^n)$ of vector-valued measurable functions on $\Omega$ is defined as $L^A(\Omega,\R^n)=(L^A(\Omega))^n$, and is equipped with the norm $\|\vec u\|_{L^A(\Omega,\R^n)}=\|\,\|\vec u\|\,\|_{L^A(\Omega)}$ for $\vec u\in L^A(\Omega,\R^n)$. The Orlicz space $L^A(\Omega,\Rnn)$ of matrix-valued measurable functions on $\Omega$ can be defined analogously.

We denote by $W^{1,A}(\Omega)$ the Orlicz-Sobolev space defined by
\begin{equation*}
W^{1,A}(\Omega):=\{u\in L^A(\Omega):\, \mbox{ $u$ is weakly differentiable and } Du\in L^A(\Omega,\R^n)\}.
\end{equation*}

The space $W^{1,A}(\Omega)$, equipped with the norm
\begin{equation*}
\|u\|_{W^{1,A}(\Omega)}:=\|u\|_{L^A(\Omega)}+\|Du\|_{L^A(\Omega,\R^n)}
\end{equation*}
is a Banach space. 

The space $WE^{1,A}(\Omega)$ is defined analogously, by replacing the space $L^A(\Omega)$ with $E^A(\Omega)$. The space $W^{1,A}_0(\Omega)$ is the closure of $C_c^\infty(\Omega)$ in the $W^{1,A}$ norm.

The Orlicz space $W^{1,A}(\Omega,\R^n)$ of vector-valued measurable functions on $\Omega$ is defined as $W^{1,A}(\Omega,\R^n)=(W^{1,A}(\Omega))^n$, and is equipped with the norm $\|\vec u\|_{W^{1,A}(\Omega,\R^n)}=\|\vec u\|_{L^A(\Omega,\R^n)}+\|D\vec u\|_{L^A(\Omega,\Rnn)}$ for $\vec u\in W^{1,A}(\Omega,\R^n)$. The analogous spaces for matrix-valued functions are defined in the same way.\\

We now introduce a notion of ordering for Young functions (see, e.g., \cite[Definition~3.5.6]{Kufn}).

Let $A,B$ be Young functions. $A$ is said to \emph{dominate} $B$, and we write $B\prec A$, if there exists a positive constant $c_0$ such that
\begin{equation}
B(t)\leq A(c_0t),\quad \mbox{for every $t>0$.}
\label{dom}
\end{equation} 
As customary, if there exists also $t_0>0$ such that \eqref{dom} holds for every $t\geq t_0$, we say that $A$ \emph{dominates} $B$ \emph{near infinity}. If $A\prec B$ and $B\prec A$, the functions $A$ and $B$ are said to be \emph{equivalent}, and we write $A\sim B$.

The following result holds (see, e.g., \cite[Theorem~3.17.1]{Kufn} and subsequent remarks).

\begin{thm}\label{inclusion}
Let $A, B$ be Young functions. Then we have:\\
(i) $L^A(\Omega)\subseteq L^B(\Omega)$ if and only if $B\prec A$;\\
(ii) $L^A(\Omega) = L^B(\Omega)$ if and only if $B\sim A$.
\end{thm}
 If $\Omega$ has finite measure, the last  condition could be replaced by $B\sim A$ near infinity.
In case $(i)$, it can be seen (e.g., \cite[Theorem~13.3]{Kras}) that there exists a constant $c>0$ such that $$\|u\|_{L^B(\Omega)}\leq c \|u\|_{L^A(\Omega)},\quad u\in L^A(\Omega).$$

Now, setting ${\tilde{A}}(t):= t\log^\beta(e+t)$, with ${\beta>0}$, and ${\tilde{B}}(t):=e^{t^{\frac{1}{\beta}}}-1$, the corresponding Orlicz spaces $L^{{\tilde{A}}}(\Omega)$ and $L^{{\tilde{B}}}(\Omega)$ are well-known in literature as Zygmund spaces (see, e.g., \cite[Section~4]{Bennett}, {\cite{Opick}}), and are denoted by $\mbox{L}{\mbox{Log}}^\beta\mbox{L}(\Omega)$ and $\mbox{Exp}_{\frac{1}{\beta}}(\Omega)$, respectively. {We will denote by $exp_\frac{1}{\beta}(\Omega)$ the closure of the class of measurable and bounded functions on $\Omega$ with respect to the norm of $\mbox{Exp}_{\frac{1}{\beta}}(\Omega)$. This space is separable, so it coincides with the closure of the smooth functions on $\Omega$ with respect to the same norm.}

{We recall here a result which clarifies the relationship between the exponential Orlicz spaces and the Lebesgue spaces (see, e.g.~\cite[Lemma~2.3 and 2.4]{MT}).}

{\begin{lem}\label{lem:inclusions}
The following hold:
\begin{itemize}
\item[(i)] $exp_s(\R^n)\not\hookrightarrow L^\infty(\R^n)$, thus ${\rm Exp}_{s}(\R^n)\not\hookrightarrow L^\infty(\R^n)$, $s\geq1$;
\item[(ii)] ${\rm Exp}_{s}(\R^n)\not\hookrightarrow L^r(\R^n)$, for all $1\leq r<s$, $s>1$;
\item[(iii)] $L^q(\R^n)\cap L^\infty(\R^n) \hookrightarrow exp_s(\R^n)$, for all $1\leq q\leq s$. Moreover,
\begin{equation*}
\|u\|_{{\rm Exp}_{s}} \leq \frac{1}{(\log 2)^s} (\|u\|_q+\|u\|_\infty)\,;
\end{equation*}
\item[(iv)] for every $1\leq s\leq q <\infty$, it holds that
\begin{equation*}
\|u\|_q\leq \left(\Gamma\left(\frac{q}{s}+1\right)\right)^\frac{1}{q}\|u\|_{{\rm Exp}_{s}},
\end{equation*}
where $\Gamma(x):=\int_0^\infty\tau^{x-1}e^{-\tau}\,\mathrm{d}\tau$, $x>0$.
\end{itemize}
\end{lem}}

We cannot find the explicit form of the complementary function ${\tilde{A}^*}$. However, from the immediate inequality
\begin{equation}
st\leq s\log^\beta(e+s) + e^{t^{\frac{1}{\beta}}}-1,\quad\mbox{for every $s,t>0$,}
\end{equation}
we deduce that ${\tilde{A}^*}\prec{\tilde{B}}$. In fact, by virtue of \cite[Theorem~6.2]{Kras}, it holds that ${\tilde{A}^*}\sim{\tilde{B}}$ near infinity. Moreover, $(\mbox{L}{\mbox{Log}}^\beta\mbox{L}(\Omega))^*=\mbox{Exp}_{\frac{1}{\beta}}(\Omega)$.

Clearly, since ${\tilde{B}}$ does not satisfy the $\Delta_2$-condition, then also ${\tilde{A}^*}$ does not.\\

In the sequel, we will use the following Poincar\'e inequality, whose proof can be found, e.g., in \cite[Prop. 2.13]{MSZ}, \cite[Lemma~5.7]{Gossez} and, in the case of the ball, e.g., in \cite[Lemma~3]{Ta1}.
\begin{prop}\label{poincare}
Let $\Omega$ be an open set of finite measure, and assume that $A$ satisfies $\Delta_2$-condition with constant $\mu$. There exists a constant $C=C(n,\mu,|\Omega|)$ such that
\begin{equation}
\int_\Omega A(|\vec u(\vec x)|)\,\dd \vec x\leq C \int_\Omega A(\|D\vec u(\vec x)\|)\,\dd \vec x,\quad \mbox{for all $\vec u\in W^{1,A}_0(\Omega,\R^n)$.}
\label{poinc}
\end{equation}
In particular, if $\Omega=B_r$ then $C=C(r)=r$ if $0<r\leq1$, while $C=C(\mu,r)\leq \mu^{\lfloor\log_2r\rfloor+1}$ if $r>1$.
\end{prop}
Recently, also the following alternative version of this inequality has been proved (see \cite[Theorem~3.9]{CIA}). It holds under the assumption that $\Omega$ has the \emph{cone property}; i.e., there exists a finite cone $P$ such that each point $x\in\Omega$ is the vertex of a finite cone $P_x$ contained in $\Omega$ and congruent to $P$.

\begin{prop}\label{poinc:2}
Let $\Omega\subset\R^n$ be an open bounded domain having the cone property, let $A$ be an $N$-function satisfying the $\Delta_2$-condition. Let $\vec u \in W^{1,A}(\Omega,\R^n)$. Then
\begin{equation*}
\int_\Omega A(\|\vec u(\vec x) - \vec u_B\|)\, \dd \vec x \leq C \int_\Omega A(\|D\vec u(\vec x)\|)\, \dd \vec x,
\end{equation*}
where 
\begin{equation*}
\vec u_B:=\fint_B \vec u(\vec y)\,\dd\vec y, 
\end{equation*}
$B$ is any ball such that $B\ssubset\Omega$ and $C$ is a positive constant depending only on $\Omega$ and $B$.
\end{prop}

Another useful tool will be the following general version of \emph{Chebyshev's inequality}: if $A$ is a non-negative and non-decreasing function defined for $t\geq0$, then for every $\lambda>0$ we have
\begin{equation}
\left|\{x\in\Omega:\, |v(\vec x)|\geq \lambda\}\right|\leq\frac{1}{A(\lambda)}\int_\Omega A(|v(\vec x)|)\, \dd \vec x.
\label{cheb}
\end{equation}
The proof of \eqref{cheb} is very simple. Denoting by $\mathbbm{1}_U$ the indicator function of set $U$ and unsing the monotonicity of $A$, we have
\begin{equation*}
\begin{split}
\left|\{x\in\Omega:\, |v(\vec x)|\geq \lambda\}\right|&=\int_\Omega \mathbbm{1}_{\{|v|\geq\lambda\}}(\vec x)\,\dd \vec x\leq \int_\Omega \mathbbm{1}_{\{A(|v|)\geq A(\lambda)\}}(\vec x)\,\dd \vec x\\
& \leq \int_{\{A(|v|)\geq A(\lambda)\}} \frac{A(|v(\vec x)|)}{A(\lambda)}\,\dd \vec x\leq\frac{1}{A(\lambda)}\int_{\Omega} A(|v(\vec x)|)\,\dd \vec x.
\end{split}
\end{equation*}

A general criterion for the equi-absolute continuity of the integrals of a family of functions in $L^A(\Omega)$ is given by the following version of \emph{Vall\'ee Poussin's Theorem} (see, e.g., \cite[Ch. II, §11.1]{Kras}):
\begin{thm}
Let $A$ be an $N$-function, and $\mathcal{F}$ be a family of functions in $L^A(\Omega)$. If there exists $C>0$ such that
\begin{equation*}
\int_\Omega A(|u(\vec x)|)\, \dd \vec x\leq C,\quad u\in\mathcal{F},
\end{equation*}
then the family $\mathcal{F}$ has equi-absolutely continuous integrals.
\label{lavpous}
\end{thm}

Let $\{\vec v_j\}$ be a sequence of functions in $L^A(\Omega,\R^n)$ and let $\vec v\in L^A(\Omega,\R^n)$. If $A$ is $\Delta_2$ near infinity, then
\begin{equation*}
\lim_{j\to+\infty}\|\vec v_j-\vec v\|_{L^A(\Omega,\R^n)}=0 \, \Leftrightarrow\, \lim_{j\to+\infty}\int_\Omega A(\|\vec v_j-\vec v\|)\,\dd \vec x=0.
\end{equation*}
Note that, if $A$ does not satisfy $\Delta_2$-condition, the implication ``$\Leftarrow$'' fails. If $A\in\Delta_2$ near infinity, instead, we have
\begin{equation*}
\lim_{j\to+\infty}\|\vec v_j-\vec v\|_{L^A(\Omega,\R^n)}=0 \, \Rightarrow\, \lim_{j\to+\infty}\int_\Omega A(\|\vec v_j\|)\,\dd \vec x=\int_\Omega A(\|\vec v\|)\,\dd \vec x.
\end{equation*}

\section{Definitions and preliminary results}\label{prelim}

This section collects some basic definitions and preliminary results.

\begin{defn}\label{Ncond} A function $\vec u:\Omega\longrightarrow\R^n$ defined everywhere satisfies \emph{Lusin's $N$ condition} if the image of a subset of $\Omega$ of measure zero is a set of measure zero. We say that $\vec u$ satisfies \emph{Lusin's $N^{-1}$ condition} if the preimage of a subset of $\R^n$ of measure zero is a set of measure zero.\end{defn}

Let $\vec u:\Omega\longrightarrow\R^n$ be a measurable function and let $\vec x_0\in\Omega$. If $\vec u$ is \emph{approximately differentiable} at $\vec x_0$, we denote by $\nabla \vec u (\vec x_0)$ its approximate differential at $\vec x_0$. We denote  the set of approximate differentiability points of $\vec u$ by $\Omega_d$.
If $\vec u$ is approximately differentiable a.e., for any $E \subset \R^n$ and $\vec y \in \R^n$, we define 
\begin{equation}
\mathcal{N}_E (\vec y):=\mathcal{H}^0(\{\vec x \in \Omega_d \cap E:\, \vec u (\vec x) = \vec y\}).
\label{N_E}
\end{equation}

Now, we recall the definition of \emph{almost everywhere (a.e.) invertibility} for a vector-valued function.

\begin{defn}\label{df:1-1ae}
A function $\vec u:\Omega\longrightarrow\R^n$ is said to be one-to-one a.e. in a subset $E\subset\Omega$ if there exists a subset $N\subset E$, with $\mathcal{L}^n(N)=0$, such that $\vec u_{|_{E\backslash N}}$ is one-to-one.
\end{defn}

The following is the notion of \emph{geometric image} of a set adapted to the context of Orlicz spaces (see \cite[Section~2.2]{HS}).

\begin{defn}\label{de:O0}
Let $\vec u\in W^{1,A}(\Omega,\R^n)$ and assume that $\mbox{det}D\vec u(\vec x)\neq0$ for a.e. $\vec x\in\Omega$. Let $\Omega_0$ be the subset of $\vec x\in\Omega$ where the following are satisfied:\\
i) $\vec u$ is approximately differentiable at $\vec x$ and $\mbox{det}\nabla \vec u(\vec x)\neq0$;
\\
ii) there exist $\vec w\in C^1(\R^n,\R^n)$ and a compact set $K\subset\Omega$ of density 1 at $\vec x$ such that $\vec u_{|_K}=\vec w_{|_K}$ and $\nabla \vec u_{|_K}=D\vec w_{|_K}$.

For any measurable $E\subset\Omega$, the geometric image of $E$ under $\vec u$ is defined as
\begin{equation}
\imG(\vec u,E):=\vec u(E\cap\Omega_0).
\label{eqn:geoimg}
\end{equation}
\end{defn}

It turns out that $\Omega_0$ is a set of full measure in $\Omega$ (see the remarks after \cite[Def.~2.4]{HS}). 

\subsection{A class of good open sets}
We consider a class of ``good'' sets, such that the restrictions of Orlicz-Sobolev functions to their boundaries enjoy some desirable properties (Definition~\ref{df:good_open_sets}(i)-(iv)).

Let be given $U \ssubset \Omega$ a nonempty, open set with a $C^2$ boundary. We call $d: \Omega \to \R$ 
the signed distance function from $U$ and consider its super-level sets 
\begin{equation}\label{eq:Ut}
 U_t := \left\{ \vec x \in \Omega : d(\vec x)>t \right\} ,
\end{equation}
for each $t \in \R$.
It is well-known (see, e.g., \cite[p.\ 112]{Sverak88} or \cite[p.\ 48]{MuSp95}) that there exists $\delta>0$ such that for all $t \in (-\delta,\delta)$, the set $U_t$ is open, $\overline{U_t}\ssubset\Omega$ and has a $C^2$ boundary.

Let $A$ be an $N$-function satisfying the growth at infinity \eqref{eq:L2logL} and the $\Delta_2$-condition 
  at infinity \eqref{delta2}. It is stated in \cite[Remark~3.2]{CarozzaCianchi19}, \cite[Prop.2.6]{HS}, that maps $\vec u \in W^{1,A}(\Omega, \R^n)$ have a continuous representative on $(n-1)$-dimensional $C^1$ manifolds. Therefore, for some $C^1$ open set $U\subset \subset \Omega$ the notation $\vec u|_{\partial U}$ will be referred to the continuous representative of $\vec u$ on $\partial U$. In addition, \cite[Prop.~2.6]{HS}, Federer's change of variables formula holds true: for any $\mathcal{H}^{n-1}$-measurable subset $E \subset \partial U$,
\begin{equation}\label{federerchange}
\mathcal{H}^{n-1}(\vec u(E))=\int_E |(\cof\nabla\vec u(\vec x)) \vecg \nu(\vec x)| \dd \mathcal{H}^{n-1}(\vec x),
\end{equation}
where $ \vecg \nu(\vec x)$ denotes the outward unit normal to $\partial U$ at $\vec x$.
\begin{defn}\label{df:good_open_sets}   
We define the class of \emph{good open sets} $\mathcal{U}_{\vec u}$ as the family of nonempty open sets $U \ssubset \Omega$ with a $C^2$ boundary where the following conditions are satisfied:
\begin{description}
\item[(i)] $\vec u|_{\partial U} \in W^{1,A}(\partial U, \R^n)$, and $(\cof \nabla \vec u)|_{\partial U} \in L^1 (\partial U, \Rnn)$;\\

\item[(ii)] $\partial U \subset \Omega_0$ $\mathcal{H}^{n-1}$-a.e., where $\Omega_0$ is the set of Definition \ref{de:O0}, and $\nabla (\vec u|_{\partial U})(\vec x) = \nabla \vec u(\vec x) |_{T_{\vec x} \partial U}$ for $\mathcal H^{n-1}$-a.e.\ $\vec x\in \partial U$, where $T_{\vec x} \partial U$ denotes the linear tangent space of $\partial U$ at $\vec x$;\\

\item[(iii)] $\displaystyle \lim_{\varepsilon \searrow 0} \fint_0^{\varepsilon} \left| \int_{\partial U_t} |\cof \nabla \vec u| \, \dd \mathcal{H}^{n-1} - \int_{\partial U} |\cof \nabla \vec u| \, \dd \mathcal{H}^{n-1} \right| \dd t = 0 $;\\

\item[(iv)] For every $\vec g \in C^1 (\R^n, \R^n)$ with $(\adj D \vec u) (\vec g \circ \vec u) \in L^1_{\mbox{loc}} (\Omega, \R^n)$,
\begin{equation*}
\begin{split}
  \lim_{\e \searrow 0} &\fint_0^{\e} \biggl| \int_{\partial U_t} \! \vec g (\vec u (\vec x)) \cdot \left( \cof \nabla \vec u (\vec x) \, \vecg \nu_t (\vec x) \right) \dd \mathcal{H}^{n-1} (\vec x) \\
&- \int_{\partial U} \! \vec g (\vec u (\vec x)) \cdot \left( \cof \nabla \vec u (\vec x) \, \vecg \nu (\vec x) \right) \dd \mathcal{H}^{n-1} (\vec x) \biggr| \dd t = 0 ,
\end{split}
\end{equation*}
where $\vecg \nu_t$ denotes the unit outward normal to $U_t$ for each $t \in (0, \e)$, and $\vecg \nu$ 
the unit outward normal to~$U$.
\end{description}
\end{defn}

\subsection{Degree for Orlicz-Sobolev maps and topological image of a set}\label{sec:degtop}

In order to introduce the concept of \emph{topological image}, we need to recall the notion of topological degree for continuous functions (see, e.g., \cite{Deimling85,FoGa95book}). 

Let $U$ be a bounded open set of $\R^n$. We've already recalled that every map $\vec u\in W^{1,A}(\partial U,\R^n)$, with $A$ verifying \eqref{eq:L2logL} and the $\Delta_2$-condition at infinity, is continuous on  $(n-1)$-dimensional $C^1$ manifolds and so admits a continuous representative $\bar{\vec u}:\partial U \longrightarrow \R^n$, that can be extended to a continuous $\tilde{\vec u}:\overline{U} \longrightarrow \R^n$ by virtue of Tietze's theorem (see, e.g., \cite[Theorem~35.1]{Mun75}). 
Therefore, the following definition of degree is consistent since the degree only depends on the boundary values (see, e.g., \cite[Th.\ 3.1 (d6)]{Deimling85}).
\begin{defn}
The degree $\deg (\bar{\vec u}, U, \cdot): \R^n \setminus \bar{\vec u} (\partial U) \to \Z$ of $\bar{\vec u}$ on $U$ is defined as the degree $\deg (\tilde{\vec u}, U, \cdot): \R^n \setminus \bar{\vec u}(\partial U) \to \Z$ of $\tilde{\vec u}$ on $U$. 
\end{defn}
  With a slight abuse of notation, we will denote by $\deg ({\vec u}, U, \cdot)$ the degree of $\vec u\in W^{1,A}(\partial U,\R^n)$, tacitly referring to the degree of its continuous representative.\\

Following the approach of \v{S}ver\'ak \cite{Sverak88} (see also \cite{MuSp95}), we are now in position to define the concept of topological image.
\begin{defn}\label{def:INV_CDL}
Let $A$ be an $N$-function satisfying \eqref{eq:L2logL}
and let $U\ssubset \R^n$ be a nonempty open set with a $C^1$ boundary.
If $\vec u \in W^{1,A}(\partial U, \R^n)$, we define $\imT(\vec u, U)$, the topological image of $U$ under 
$\vec u$, as the set of $\vec y \in \R^n \setminus \vec u (\partial U)$ such that $\deg(\vec u, U, \vec y) \neq 0$.
\end{defn}

The continuity of function $\deg (\vec u, U,\cdot)$ implies that the set $\imT(\vec u, U)$ is open and $\partial \imT(\vec u, U) \subset \vec u(\partial U)$. In addition, as $\deg(\vec u, U, \cdot)=0$ in the unbounded component of $\R^n \setminus \vec u (\partial U)$ (see, e.g., \cite[Sect.\ 5.1]{Deimling85}), it follows that $\imT (\vec u, U)$ is bounded.

\section{The class $\mathcal{A}(\Omega)$ of admissible functions} \label{sec:admissible}

{First, we denote by $\A(\Omega)$  the class } of admissible deformations consisting, roughly speaking, of Sobolev-Orlicz functions which are orientation preserving and such that no cavitation occurs.

{From now on, we fix as $N$-function $A$ satisfying \eqref{eq:L2logL} and the $\Delta_2$-condition at infinity \eqref{delta2} the function $A(t):= t^{n-1}\log^\alpha(e+t)$ for $\alpha\in(n-2,n-1)$.}

\begin{defn}\label{energydefn}
Let $\vec u\in W^{1,A}(\Omega,\R^n)$ and assume that $\mbox{det} D\vec u\in L^1(\Omega)$. For {$\phi\in C^1_c(\Omega)$} and $\vec g\in C^1_c(\R^n,\R^n)$ we define the energy
\begin{equation}
\mathcal{E}_\Omega(\vec u,\phi,\vec g):=\int_\Omega \bigl[\mbox{cof}\,D\vec u(\vec x)\cdot(\vec g(\vec u(\vec x))\otimes D\phi(\vec x))+\mbox{det}D\vec u(\vec x)\phi(\vec x)\mbox{div}\vec g(\vec u(\vec x))\bigr]\,\mathrm{d}\vec x.
\label{energy}
\end{equation}
\end{defn}

\begin{oss}
{
We notice that if $\vec u\in W^{1,A}(\Omega,\R^n)$, then $D\vec u\in L^A(\Omega,\Rnn)$, so $\cof D\vec u\in \mbox{LLog}^{\frac{\alpha}{n-1}}\mbox{L}(\Omega,\Rnn)$. In particular, $\cof D\vec u\in L^1_{\rm loc}(\Omega,\Rnn)$. This implies that the energy \eqref{energy} is finite.}
\end{oss}

\begin{defn}\label{def:SE}
Let $\vec u : \Omega \longrightarrow \R^n$ be measurable and approximately differentiable a.e.
Assume that $\det \nabla \vec u \in L^1_{\mbox{loc}} (\Omega)$ and $\cof \nabla \vec u \in L^1_{\mbox{loc}} (\Omega,\Rnn)$.

For every $\vec f \in C^1_c (\Omega \times \R^n,\R^n)$, define
\begin{equation}\label{eq:Euf}
 \bar{\mathcal{E}}_\Omega (\vec u, \vec f) := \int_{\Omega} \left[ \cof \nabla \vec u (\vec x) \cdot D \vec f (\vec x, \vec u (\vec x)) + \det \nabla \vec u (\vec x)\, \mbox{div}\,\vec f (\vec x, \vec u (\vec x))  \right] \dd \vec x
\end{equation}
and
\[
 \bar{\mathcal{E}}_\Omega (\vec u) := \sup \biggl\{ \bar{\mathcal{E}}_\Omega (\vec u, \vec f) : \ \vec f \in C^1_c (\Omega \times \R^n,\R^n), \ \| \vec f \|_{\infty} \leq 1 \biggr\} .
\]
\end{defn}
In equation \eqref{eq:Euf}, $D \vec f (\vec x, \vec y)$ denotes the derivative of $\vec f (\cdot, \vec y)$ evaluated at $\vec x$, while $\mbox{div}\,\vec f (\vec x, \vec y)$ is the divergence of $\vec f (\vec x, \cdot)$ evaluated at $\vec y$.

Note that if we restrict the definition of $\mathcal{E}_\Omega(\vec u,\phi,\vec g)$ to test functions $\phi\in C^1_c(\Omega)$, then defining $\vec f(\vec x,\vec y):=\phi(\vec x)\vec g(\vec y)$ we have $\mathcal{E}_\Omega(\vec u,\phi,\vec g)=\bar{\mathcal{E}}_\Omega (\vec u, \vec f)$. Therefore, passing to the supremum: $\mathcal{E}_\Omega(\vec u):= \sup \{\mathcal{E}_\Omega(\vec u,\phi,\vec g):\, \|\phi \vec g\|\le 1\}$. Along the lines of the proof of \cite[Theorem~4.6]{HM11}, it can be shown that $\mathcal{E}_\Omega(\vec u)=\bar{\mathcal{E}}_\Omega (\vec u)$ for all functions $\vec u\in W^{1,A}(\Omega,\R^n)$ such that $\mbox{det} D\vec u\in L^1(\Omega)$ and $\mbox{det} D\vec u>0$ a.e.

The energy $\bar{\mathcal{E}}_\Omega (\vec u)$ was introduced in \cite{HM10} and measures the new surface in the deformed configuration created by $\vec u$. For our purposes, we are interested into deformations $\vec u$ such that $\bar{\mathcal{E}}_\Omega (\vec u)=0$; i.e., that do not exhibit cavitation.

It is useful for the sequel to recall the definition of distributional determinant (see, e.g., \cite{Ball1}).
In the expression below, the symbol $\langle \cdot , \cdot \rangle$ denotes the duality product between a distribution and a smooth function.

\begin{defn}\label{de:Det}
Let $\vec u \in W^{1,1} (\Omega, \R^n)$ satisfy $(\adj D \vec u) \, \vec u \in L^1_{\mbox{loc}} (\Omega, \R^n)$.
The distributional determinant of $\vec u$ is the distribution $\Det D \vec u$ defined as
\[
 \langle \Det D\vec u , \phi \rangle := -\frac{1}{n} \int_\Omega \vec u(\vec x) \cdot (\cof D\vec u(\vec x)) \, D\phi(\vec x) \, \dd \vec x, \qquad \phi\in C_c^\infty(\Omega) .
\]
\end{defn}

The equality $\Det D\vec u= \det D\vec u$, when $\Det D\vec u\in L^1(\Omega)$, can be intended as
\begin{equation*}
-\frac{1}{n} \int_\Omega \vec u(\vec x) \cdot (\cof D\vec u(\vec x)) \, D\phi(\vec x) \, \dd \vec x=\int_\Omega\det D\vec u(\vec x)\phi(\vec x)\, \dd \vec x, \qquad \phi\in C_c^\infty(\Omega) .
\end{equation*}

We introduce the class $\mathcal{A}(\Omega)$ of \emph{admissible functions} as follows.

\begin{defn}
A function $\vec u\in W^{1,A}(\Omega,\R^n)$ is said to be \emph{admissible}, and we write $\vec u\in\mathcal{A}(\Omega)$, if $\mbox{det} D\vec u\in L^1(\Omega)$, $\mbox{det} D\vec u>0$ a.e. and $\mathcal{E}_\Omega(\vec u)=0.$
\label{admissible}
\end{defn}

For each such an admissible deformation $\vec u$, $\bar{\mathcal{E}}_\Omega(\vec u)={\mathcal{E}}_\Omega(\vec u)=0$. Furthermore, the conditions $\bar{\mathcal{E}}_\Omega(\vec u)=0$ and $\det D\vec u=\Det D\vec u$ are equivalent, as expressed by the following theorem.

\begin{thm}{\rm \cite[Theorem~1.1]{HS}}\label{thm:StroffoHenao} Let $A$ be a Young function satisfying \eqref{eq:L2logL} and 
assume that $\vec u\in W^{1,A}(\Omega. \R^n)$ satisfies $\det D\vec u\in L^1_{loc}(\Omega)$. Then the following are equivalent:
\begin{itemize}
\item{}$\bar{\mathcal{E}}_\Omega(\vec u)=0$ and $\det D\vec u>0$ a.e.; 
\item{}$(\adj D\vec u)\vec u\in L^1_{loc}(\Omega. \R^n)$, $\det D\vec u(\vec x)\not=0$ for a.e.\ $\vec x\in \Omega$, 
  $\det D\vec u=\Det D\vec u$ and $\deg (\vec u, B(\vec x, r))\ge 0$ for every $\vec x\in \Omega$ and a.e.\ $r\in (0, \dist(\vec x, 
  \partial \Omega))$.
\end{itemize}
\end{thm}

\subsection{Some properties of class $\A(\Omega)$: fine properties, openness and local invertibility}\label{sec:fine}

In this section, we preliminarly recall some fine properties for admissible deformations $\vec u\in\A(\Omega)$ (\cite[Proposition~4.2]{HS}). Here, $\mathcal{U}_{\vec u}$ is the class of good open sets introduced with Def.~\ref{df:good_open_sets}, and $\mathcal{N}_U$ is the number defined by \eqref{N_E}.

\begin{prop}
 \label{th:char1}
Let $\vec u \in \mathcal A(\Omega)$. Then the following properties hold:
  \begin{description}
   \item[(i)] $\vec u \in L^\infty_{{\rm loc}}(\Omega, \R^n)$;\\
   \item[(ii)] $\Det D \vec u = \det D \vec u$;\\
   \item[(iii)] For all $U \in \mathcal{U}_{\vec u}$,
 \begin{equation}\label{eq:degUNU}
 \deg (\vec u, U, \cdot) = \mathcal{N}_U \ \text{ a.e.}
 \quad \text{and}\quad \imT(\vec u, U)=\imG(\vec u, U)\ \text{a.e.};\\
  \end{equation}
   \item[(iv)] For every $U_1$, $U_2\in \mathcal{U}_{\vec u}$ with $U_1\ssubset U_2$,
   \begin{equation*}
    \label{pr:43d}
 \deg (\vec u, U_1, \cdot) \leq \deg (\vec u, U_2, \cdot) \text{ a.e.\ and in } \R^n \setminus \vec u (\partial U_1 \cup \partial U_2),
\end{equation*}
and
\begin{equation*}
\overline{\imT(\vec u, U_1)} \subset \overline{\imT(\vec u, U_2)};
\end{equation*}
   \item[(v)] The components of $\vec u$ are weakly 1-pseudomonotone.
  \end{description}
\end{prop}
\begin{defn}
A function $ u\in W^{1,1}(\Omega)$ is called $K$-weakly pseudomonotone if for every $x\in \Omega$ and a.e. $0<r<\dist(x, \partial \Omega)$,
\begin{equation*}
\displaystyle\mathop{\ess \osc}_{B(x,r)} u\le K\mathop{\ess \osc}_{S(x,r)} u
\end{equation*}
Notice that the oscillation on the left is meant to be essential with respect to the Lebesgue measure while on the right with respect to the $(n-1)$-Hausdorff measure.
\end{defn}

\begin{defn}\label{de:imTux}
Let $\vec u \in \mathcal{A}(\Omega)$.
We define the \emph{topological image} of a point $\vec x \in \Omega$ by $\vec u$ as 
\begin{equation*}
 \imT(\vec u, \vec x) := \bigcap_{\substack{r>0 \\ B(\vec x, r) \in \mathcal U_{\vec u}}} \overline{\imT(\vec u, B(\vec x,r))},
\end{equation*}
and $NC:=\{\vec x\in\Omega \colon \mathcal{H}^0 (\imT(\vec u,\vec x)) >1\}$.
\end{defn}

It's worth noting that both the definitions of $\imT(\vec u, \vec x)$ and $NC$ do not depend on the particular representative of $\vec u$ (see \cite[Remark\ 5.7.(c)]{HBMC} for explanations). 

\begin{prop}{\rm \cite[Proposition~4.5]{HS}} 
 \label{pr:fine}
For every $\vec u \in \mathcal{A}(\Omega)$ the following are satisfied:
\begin{description}
 \item[(i)] $\mathcal{H}^1(NC)=0$;
 \item[(ii)] For every $\vec x_0 \in \Omega \setminus NC$, 
\begin{equation*}
\lim_{r\searrow 0}  \fint_{B(\vec x_0, r)} \vec u(\vec x)\,\dd \vec x=:\vec u^*(\vec x_0)\in \R^n;
\end{equation*} 
 \item[(iii)] The map $\hat{\vec u}$ defined \emph{everywhere} in $\Omega$ by
 \begin{equation}
 \hat {\vec u}(\vec x):=
\begin{cases} 
  \vec u^*(\vec x) & \text{if }
  \vec x \in \Omega \setminus NC, \\ 
\text{any element of $\imT(\vec u,\vec x)$} & \text{if } \vec x \in NC  
\end{cases}
\end{equation}
 is such that $\hat{\vec u}(\vec x)=\vec u(\vec x)$ for every $\vec x \in \Omega_0$ and
 $\hat {\vec u}(\vec x)\in \imT(\vec u, \vec x)$ for every $\vec x\in \Omega$.
 Moreover, it is continuous at every point of $\vec x \in \Omega\setminus NC$,
 differentiable a.e., and such that $\mathcal{L}^n(\hat {\vec u}(N))=0$
 for every $N\subset \Omega$ with $\mathcal{L}^n(N)=0$.
\end{description}

\end{prop}

Note that  equality \eqref{eq:degUNU} in Prop.~\ref{th:char1} implies an \emph{openness} property for $\vec u$:
		for every $U\in \mathcal{U}_{\vec u}$,
		\begin{equation}\label{eq:imTUimGU}
		 \imT (\vec u, U) = \imG (\vec u, U) \quad \text{a.e.}
		\end{equation}
		
\begin{defn} \label{de:imTomega}
Let $\vec u\in \mathcal A(\Omega)$. Define $$\mathcal{U}^N_{\vec u} := \left\{ U \in \mathcal{U}_{\vec u}:\, \partial U \cap NC = \varnothing \right\}$$
and
\[
 \imT (\vec u, \Omega) := \bigcup_{U\in\mathcal U^N_{\vec u}} \imT (\vec u, U).
\]
\end{defn}

We will see in Section \ref{complow} that $\imT (\vec u, \Omega)$ plays the role of the \emph{deformed configuration}.
By the continuity of the degree, $\imT (\vec u, U)$ is open, and hence, so is $\imT (\vec u, \Omega)$.
Moreover, it does not depend on the particular representative of $\vec u$ (\cite[Lemma 5.18.(b)]{HBMC}).\\

We recall here some results of local invertibility for functions $\vec u\in\A(\Omega)$.

\begin{defn} \label{de:Oinv}
Let $\vec u\in \mathcal A(\Omega)$.
We denote by $\mathcal{U}^{\mbox{in}}_{\vec u}$ the class of $U\in \mathcal U_{\vec u}$ such that $\vec u$ is one-to-one a.e.\ in $U$
(see Definition \ref{df:1-1ae}), 
and by $\mathcal{U}^{N,\mbox{in}}_{\vec u}$ the set $\mathcal{U}^N_{\vec u} \cap \mathcal{U}^{\mbox{in}}_{\vec u}$. 
Define $$\Omega_{\mbox{in}}:=\bigcup \{U:\, U\in\mathcal{U}^{\mbox{in}}_{\vec u}\}.$$
\end{defn}

The set $\Omega_{\mbox{in}}$
consists of the sets of points around which $\vec u$ is locally a.e.\ invertible: 
$\vec x \in \Omega_{\mbox{in}}$ if and only if there exists $r>0$ such that $\vec u$ is one-to-one a.e.\ in $B (\vec x, r)$.
It does not depend on the particular representative of $\vec u$ and it turns out that $\Omega_{\mbox{in}}$ is of full measure in $\Omega$ (see \cite[Proposition~4.9]{HS}).

Equality \eqref{eq:imTUimGU} allows us to define the local inverse having for domain the open set $\imT (\vec u, U)$.
\begin{defn}\label{de:inverse}
Let $\vec u \in \mathcal{A}(\Omega)$ and $U \in \mathcal{U}^{\mbox{in}}_{\vec u}$.
The inverse $(\vec u|_U)^{-1} : \imT (\vec u, U) \to \R^n$ is defined a.e.\ as $(\vec u|_U)^{-1} (\vec y) = \vec x$, 
for each $\vec y \in \imG (\vec u, U)$, and where $\vec x \in U \cap \Omega_0$ satisfies $\vec u (\vec x) = \vec y$.
\end{defn}

The following results hold (see \cite[Propositions~4.11 and 4.12]{HS}).

\begin{prop}\label{prop:locINV}
Let $\vec u \in \mathcal{A}(\Omega)$ and $U \in \mathcal{U}^{\emph{in}}_{\vec u}$. Then 
\[
(\vec u|_U)^{-1} \in W^{1,1} (\imT (\vec u, U), \R^n) \quad \text{and} 
\quad D (\vec u|_U)^{-1} = \left( D \vec u \circ (\vec u|_U)^{-1} \right)^{-1} \text{ a.e.}
\]
\end{prop}

\begin{prop}\label{th:ujINV}
For each $j \in \N$, let $\vec u_j, \vec u \in \mathcal{A}(\Omega)$ satisfy $\vec u_j \weakc \vec u$ in $W^{1,A} (\Omega, \R^n)$ as $j \to \infty$.
The following assertions hold:
\begin{description}
\item[(i)]\label{item:thINV2} For any $U \in \mathcal{U}^N_{\vec u}$ and any 
compact set $K \subset \imT (\vec u, U)$ there exists a subsequence for which 
	${K} \subset \imT (\vec u_j, \Omega)$ for all $j \in \N$.
\item[(ii)]\label{item:thINV5} For a subsequence, there exists a disjoint family
\[
\{ B_k \}_{k \in \N} \subset \biggl(\mathcal{U}^{N,\emph{in}}_{\vec u} \cap \bigcap_{j \in \N} \mathcal{U}^{N,\emph{in}}_{\vec u_j}\biggr)
\]
such that $\Omega = \bigcup_{k \in \N} B_k$ a.e.\ and, for each $k \in \N$,
\begin{equation}\label{eq:Bkinv}
 \vec u_j \to \vec u\quad  \text{uniformly on } \partial B_k, \text{ as } j \to \infty .
\end{equation}

\item[(iii)]\label{item:thINV3} Let $B \in \mathcal{U}^{\emph{in}}_{\vec u} \cap \bigcap_{j \in \N} \mathcal{U}^{\emph{in}}_{\vec u_j}$ and take an open set $V \ssubset \imT (\vec u, B)$ such that $V \subset \imT (\vec u_j, B)$ for all $j \in \N$.
Then
\begin{description}
\item[(a)]\label{item:ujINVi} $(\vec u_j|_B)^{-1} \weakcs (\vec u|_B)^{-1}$ in $BV (V, \R^n)$ as $j \to \infty$;

\item[(b)]\label{item:ujINViii} for any minor $M$, we have $M (D (\vec u_j|_B)^{-1})$, $M (D (\vec u|_B)^{-1})\in L^1 (V)$ for all $j \in \N$ and 
\[
 M\left( D (\vec u_j|_B)^{-1} \right) \weakcs M \left( D (\vec u|_B)^{-1} \right) \quad \text{in } \mathcal{M} (V) 
\text{ as } j \to \infty .
\]
\end{description}
If, in addition, the sequence $\{ \det D (\vec u_j|_B)^{-1} \}_{j \in \N}$ is equiintegrable in $V$, then the convergence in {\bf (iii)-(a)} holds in the weak topology of $W^{1,1} (V, \R^n)$, and the convergence in {\bf (iii)-(b)} holds in the weak topology of $L^1 (V)$.

\item[(iv)]\label{item:thINV4} For a subsequence we have that $\chi_{\imT (\vec u_j, \Omega)} \to \chi_{\imT (\vec u, \Omega)}$ a.e.\ and in $L^1 (\R^n)$ as $j \to \infty$.

\end{description}
\end{prop}

\subsection{{The subclass $\mathcal{A}_\beta(\Omega)$:} cut-and-paste and composition with Lipschitz functions}\label{sec:cut}


{In order to prove the semicontinuity result, we need to establish some stability properties with respect to cut-and-paste operations and the composition with suitable smooth functions. To this aim, we restrict ourselves to the subclass $\mathcal{A}_{\beta}(\Omega)$ of $\mathcal{A}(\Omega)$ where we require, in addition, that
\begin{equation*}
\cof D\vec u\in \frac{\mbox{L}^{\frac{p}{n-1}}}{\mbox{Log}^{\beta}\mbox{L}}(\Omega,\Rnn)
\end{equation*}
where $p>n-1, \beta>0$.  We notice that we are requiring a smaller space for the integrability of the cofactor. Consequently, the space of test functions $\phi$ is  $W^1 {\mbox L}^{\frac{np}{n-p+1}} \mbox{Log}^{\beta}{\mbox L}(\Omega)$, that is larger than the Sobolev-Orlicz space associated to the function $e^{t^{\frac{n-1}{\alpha}}}-1$. For relations of Lebesgue (and Orlicz) with $exp_{\frac{n-1}{\alpha}}(\Omega)$ see Lemma~\ref{lem:inclusions}.}  \par
A first result - showing that when we glue two functions in the class $\mathcal{A}_{{\beta}}(\Omega)$ that coincide in a neighborhood of an open set, the resulting function is also in $\mathcal{A}_{{\beta}}(\Omega)$ - is essentially a rewriting of \cite[Lemma~3.8]{MCOl}. Indeed, the only condition to check is that $\mathcal{E}_\Omega(\vec w)=0$ (see the definition of $\vec w$ below), and this property only involves the test functions. Therefore the proof will be omitted.
\begin{lem}
Let $U,U'$ be open sets such that $U'\ssubset U\ssubset\Omega$, and let $\vec u\in\mathcal{A}_{{\beta}}(\Omega)$, $\vec v\in\mathcal{A}_{{\beta}}(U)$ satisfy $\vec u=\vec v$ a.e. in $U\backslash U'$. Then the function
\begin{equation*}
\vec w(\vec x):=
\begin{cases}
\vec v(\vec x), & \mbox{ if $\vec x\in U'$}\\
\vec u(\vec x), & \mbox{ if $\vec x\in \Omega\backslash U'$}
\end{cases}
\end{equation*}
belongs to $\mathcal{A}_{{\beta}}(\Omega)$.
\end{lem}

Now, along the lines of the proof of \cite[Lemma~3.9]{MCOl}, we prove that $\mathcal{E}_\Omega(\vec u,\phi,\vec g)=0$ if $\vec u\in\mathcal{A}_{{\beta}}(\Omega)$ and {$\phi\in W^{1}_0\mbox{L}^{\frac{np}{n-p+1}} \mbox{Log}^{\beta}\mbox{L}(\Omega)$.}

\begin{lem}\label{lemzero}
Let $\vec u\in\mathcal{A}(\Omega)$, $\vec g\in C^1_c(\R^n,\R^n)$ and {$\phi\in W^{1}_0{\rm L}^{\frac{np}{n-p+1}} {\rm Log}^{\beta}{\rm L}(\Omega)\cap L^\infty(\Omega)$}. Then $\mathcal{E}_\Omega(\vec u,\phi,\vec g)=0$.
\end{lem}

\proof

It will suffice to construct a sequence of functions $\phi_j$ in $C^1_c(\Omega)$ such that $\phi_j\to\phi$ in {$W^{1}{\rm L}^{\frac{np}{n-p+1}} {\rm Log}^{\beta}{\rm L}(\Omega)$} and $\phi_j\rightharpoonup^*\phi$ in $L^\infty(\Omega)$ as $j\to+\infty$. Indeed, by Def.~\ref{admissible} we have that $\mathcal{E}_\Omega(\vec u,\phi_j,\vec g)=0$ for every $j\in\N$ and, as a consequence of H\"older inequality,
\begin{equation}
\begin{split}
&\left|\int_\Omega \mbox{cof}\,D\vec u(\vec x)\cdot(\vec g(\vec u(\vec x))\otimes D\phi_j(\vec x))\,\mathrm{d}\vec x - \int_\Omega \mbox{cof}\,D\vec u(\vec x)\cdot(\vec g(\vec u(\vec x))\otimes D\phi(\vec x))\,\mathrm{d}\vec x\right|\\
&=\left|\int_\Omega \mbox{cof}\,D\vec u(\vec x)\cdot(\vec g(\vec u(\vec x))\otimes (D\phi_j(\vec x)-D\phi(\vec x))\,\mathrm{d}\vec x\right|\\
&\leq C \|\mbox{cof}\,D\vec u\|_{{\frac{{\rm L}^{\frac{p}{n-1}}}{{\rm Log}^{\beta}{\rm L}}(\Omega,\Rnn)}}\|D\phi_j-D\phi\|_{{{\rm L}^{\frac{np}{n-p+1}}{\rm Log}^{\beta}{\rm L}(\Omega,\R^n)}}\to0.
\end{split}
\end{equation}
Moreover, since $\mbox{det}D\vec u\in L^1(\Omega)$ and $\mbox{div}(\vec g\circ\vec u)$ is bounded, we have
\begin{equation}
\begin{split}
&\left|\int_\Omega \mbox{det}D\vec u(\vec x)\phi_j(\vec x)\mbox{div}\vec g(\vec u(\vec x))\,\mathrm{d}\vec x - \int_\Omega \bigl[\mbox{det}D\vec u(\vec x)\phi(\vec x)\mbox{div}\vec g(\vec u(\vec x))\bigr]\,\mathrm{d}\vec x\right|\\
&= \left|\int_\Omega \mbox{det}D\vec u(\vec x)\mbox{div}\vec g(\vec u(\vec x))(\phi_j(\vec x)-\phi(\vec x))\,\mathrm{d}\vec x\right|\to0.
\end{split}
\end{equation}

\endproof

With the following lemma, we prove that the class $\mathcal{A}_{{\beta}}(\Omega)$ is stable under the composition with suitable Lipschitz functions.

\begin{lem}\label{lem3.10}
Let $\vec u\in\mathcal{A}_{{\beta}}(\Omega)$, $B\ssubset\Omega$ a ball, $\vecg\rho:B\longrightarrow\overline{B}$ Lipschitz such that $\vecg\rho_{|_{\partial B}}={\vec {id}}_{|_{\partial B}}$, $\emph{det}D\vecg\rho>0$ a.e. and
\begin{equation}
{\int_B\log^\beta\left(e+\frac{1}{\det D\vecg\rho(\vec x)}\right)(\det D\vecg\rho(\vec x))^{-\frac{(p-1)(n+1)}{n-p+1}}\,\dd \vec x <+\infty.} 
\label{conddet}
\end{equation}
Define 
\begin{equation}
\vec z:=
\begin{cases}
\vec u\circ\vecg\rho & \mbox{ in $B$,}\\
\vec u & \mbox{ in $\Omega\backslash B$.}
\end{cases}
\end{equation}
Assume that $\vec z\in W^{1,A}(\Omega,\R^n)$, $D\vec z=(D\vec u\circ\vecg\rho)D\vecg\rho$ in $B$ and $\emph{det}D\vec z\in L^1(\Omega)$. Then $\vec z\in\mathcal{A}_{{\beta}}(\Omega)$.
\end{lem}

\proof
We can perform a similar argument as for \cite[Lemma~3.10]{MCOl}, exploiting the change of variable formula, the chain rule and the fact that $\vecg\rho^{-1}\in W^{1,1}(B)$ (\cite[Theorem~8]{Sverak88}). We should only check that if $\phi\in C^1_c(\Omega)$, then {$\phi\circ\vecg\rho^{-1}\in W^{1}{\rm L}^{\frac{np}{n-p+1}} {\rm Log}^{\beta}{\rm L}(B)$}. Indeed, {setting $\tilde{A}(t):=t^{\frac{np}{n-p+1}}\log^\beta(e+t)$}, we have

\begin{equation}
\begin{split}
&\int_B{\tilde{A}}(\|D(\phi\circ\vecg\rho^{-1})(\vec y)\|)\,\dd \vec y = \int_B{\tilde{A}}(\|D\phi\circ\vecg\rho^{-1}(\vec y)\|\|D\vecg\rho^{-1}(\vec y)\|)\,\dd \vec y\\
&\leq  c_{\phi,\rho}\int_B{\tilde{A}}(\|D\vecg\rho^{-1}(\vec y)\|)\,\dd \vec y = c_{\phi,\rho}\int_B{\tilde{A}}(\|D\vecg\rho^{-1}(\vecg\rho(\vec x))\|)\det D\vecg\rho(\vec x)\,\dd \vec x\\
& = c_{\phi,\rho}\int_B{\tilde{A}}(\|\cof D\vecg\rho(\vec x)\|(\det D\vecg\rho(\vec x))^{-1})\det D\vecg\rho(\vec x)\,\dd \vec x\\
& \leq c'_{\phi,\rho}\int_B{\tilde{A}}((\det D\vecg\rho(\vec x))^{-1})\det D\vecg\rho(\vec x)\,\dd \vec x <+\infty,
\end{split}
\label{stimata}
\end{equation}
whence {$D(\phi\circ\vecg\rho^{-1})\in{\rm L}^{\frac{np}{n-p+1}} {\rm Log}^{\beta}{\rm L}(B,\R^n)$} under assumption \eqref{conddet}. In \eqref{stimata}, the constants $c_{\phi,\rho}, c'_{\phi,\rho}$ depend on ${\tilde{A}}$, $\|D\phi\circ\vecg\rho^{-1}\|_\infty$ and $\|\cof D\vecg\rho\|_\infty$. Now, setting $\tilde{\phi}=\phi\circ\vecg\rho^{-1}$ on $B$, $\tilde{\phi}=\phi$ on $\Omega\backslash B$, since $\vecg\rho_{|_{\partial B}}={\vec {id}}_{|_{\partial B}}$ we have that {$\tilde{\phi}\in W_0^{1}{\rm L}^{\frac{np}{n-p+1}} {\rm Log}^{\beta}{\rm L}(\Omega)$} and, in view of Lemma~\ref{lemzero}, $\mathcal{E}_\Omega(\vec z,{\phi},\vec g)=\mathcal{E}_\Omega(\vec u,\tilde{\phi},\vec g)=0$. Thus, $\vec z\in\A_{{\beta}}(\Omega)$.

\endproof


\subsection{Polyconvexity, quasiconvexity and tangential quasiconvexity}\label{sec:polyconv}

In this section, we recall the definitions of polyconvexity, quasiconvexity and tangential quasiconvexity (see, e.g., \cite{Ball1, BaMu84, Daco, DaFoMaTr99}) adapted to our setting.

Let $\tau$ be the number of minors of an $n \times n$ matrix, and denote by $\vec M(\vec F) \in \R^{\tau}$ the collection of all the minors of an $\vec F \in \Rnn$ in a given order such that its last component is $\det \vec F$. 
\begin{defn}
A Borel function $W: \Rnn \to \R \cup \{\infty\}$ is polyconvex if there exists a convex function $\Phi : \R^{\tau} \to \R \cup \{\infty\}$ such that $W(\vec F) = \Phi (\vec M(\vec F))$ for all $\vec F \in \Rnn$.
\end{defn}

We recall the classical concept of \emph{quasiconvexity}, adapted to the case of functions that can take infinite values (see, e.g., \cite{BaMu84}).

\begin{defn}\label{de:quasiconvex}
A Borel function $W: \Rnn \to \R \cup \{\infty\}$ is quasiconvex if for all $\vec F \in \Rnn$ and all $\vecg\varphi \in W^{1, \infty} (B_1, \R^n)$ with $\vecg\varphi (\vec x) = \vec F\vec x$ on $\partial B_1$ in the sense of traces, we have
\[
 W(\vec F) \leq \fint_{B_1} W(D\vecg\varphi(\vec x)) \, \dd \vec x .
\]
\end{defn}

As for the definition of \emph{quasiconvexification}, the natural one corresponding to Definition \ref{de:quasiconvex} is given in \cite[Definition~2.3]{CoDo15} and reads as follows.

\begin{defn}\label{de:quasiconvexification}
The quasiconvexification $W^{qc}: \Rnn \to \R \cup \{\infty\}$ of a Borel function $W : \Rnn \to \R \cup \{\infty\}$
is defined as
\[
 W^{qc} (\vec F) := \inf \left\{\fint_{B_1} W(D\vecg\varphi(\vec x)) \, \dd \vec x: \, \vecg\varphi\in W^{1,\infty}(B_1,\R^{n}), \, \vecg\varphi(\vec x)=\vec F\vec x \text{ on } \partial B_1\right\} .
\] 
\end{defn}

For functions $W=W(\cdot,\vec n): \Rnn \times \mS^{n-1} \to \R \cup \{\infty\}$, the definitions of polyconvexity, quasiconvexity and of quasiconvexification $W^{qc}$ always refer to the first variable. It can be proved that if $W: \Rnn \times \mS^{n-1} \to \R \cup \{\infty\}$ is quasiconvex, then it is continuous (see \cite[Proposition~2.4]{MCOl} for details).
\begin{prop}\label{pr:Wqccont}
Assume that $W: \R^{n\times n}_{+} \times \mS^{n-1} \to [0,\infty)$ is continuous and there exists an $\omega:[0,2] \to [0,\infty)$ with $\lim_{t\to 0} \omega(t)=0$ such that for all $\vec F \in \Rnn_+$ and $\vn,\vec m \in \mS^{n-1}$,
\begin{equation}\label{eq:modcontW}
 \left| W(\vec F,\vn)-W(\vec F,\vec m) \right|\leq \omega \left( \|\vn-\vec m\| \right) W(\vec F,\vn)  .
\end{equation}
Consider the extension of $W$ by infinity outside $\R^{n\times n}_{+} \times \mS^{n-1}$.
Then $W^{qc}|_{\R^{n\times n}_{+} \times \mS^{n-1}}$ is continuous.
\end{prop}

To conclude this section, we introduce the concept of \emph{tangential quasiconvexity} and \emph{tangential quasiconvexification}, where the term ``tangential'' is referred to the manifold $\mS^{n-1}$.

For each $\vec z \in \mS^{n-1}$ we denote the tangent space of $\mS^{n-1}$ at $\vec z$ by $T_{\vec z} \mS^{n-1}$.
Given an Orlicz-Sobolev function $\vn$ defined in an open set $U \subset \R^n$ such that $\vn (\vec y) \in \mS^{n-1}$ for a.e.\ $\vec y \in U$, it holds that $D \vn (\vec y) \in (T_{\vn(\vec y)} \mS^{n-1})^n$ for a.e.\ $\vec y \in U$.
Therefore, a function $V=V(\vn(\vec y),D\vn(\vec y))$ 
need only be defined in
\[
 T^n \mS^{n-1} := \left\{ (\vec z, \vecg\xi) : \,\vec z \in \mS^{n-1}, \, \vecg\xi \in (T_{\vec z} \mS^{n-1})^n \right\} .
\]
Thus, we consider a Borel function $V : T^n \mS^{n-1} \to [0, \infty)$.
The following definition (see \cite[Def. 2.6]{MCOl}) extends the one given in \cite{DaFoMaTr99} when $V$ does not depend on the first variable.
\begin{defn}\label{defn:tangq}
Let $V : T^n \mS^{n-1} \to [0, \infty)$ be a Borel function.
\begin{enumerate}
\item[(a)]
$V$ is tangentially quasiconvex if for all $(\vec z, \vecg\xi) \in T^n \mS^{n-1}$ and all $\vecg\varphi \in W^{1,\infty} (B_1, T_{\vec z}\mS^{n-1})$ with $\vecg\varphi (\vec y) = \vecg\xi \vec y$ on $\partial B_1$ in the sense of traces we have
\[
 V (\vec z,\vecg\xi) \leq \fint_{B_1} V(\vec z, D\vecg\varphi(\vec y))\, \dd \vec y .
\]
\item[(b)]
The tangential quasiconvexification $V^{tqc} : T^n \mS^{n-1} \to [0, \infty)$ of $V$ is
\begin{align*}
 & V^{tqc} (\vec z,\vecg\xi) \\
 & := \inf \left\{ \fint_{B_1} \! V(\vec z, D\vecg\varphi(\vec y))\, \dd \vec y : \, \vecg\varphi\in W^{1,\infty} (B_1, T_{\vec z}\mS^{n-1}) , \, \vecg\varphi (\vec y) = \vecg \xi \vec y \text{ on } \partial B_1 \right\}.
\end{align*}
\end{enumerate}
\end{defn}

Note that the fact $\vecg\varphi \in W^{1,\infty} (B_1, T_{\vec z}\mS^{n-1})$ implies $D\vecg\varphi (\vec y) \in (T_{\vec z}\mS^{n-1})^n$ for a.e.\ $\vec y\in B_1$.
From the definitions, it is immediate to check that $V^{tqc}$ is tangentially quasiconvex and that $V$ is tangentially quasiconvex if and only if $V = V^{tqc}$.\\

The next theorem extends to the Orlicz-Sobolev setting the main results of \cite{DaFoMaTr99}, obtained under standard $p$-growth assumptions; again, the formulation is adapted to cover a dependence of $V$ on the first variable as well.

First, we note that an explicit formula for tangential quasiconvexification in the case of the unit sphere $\mS^{n-1}$ has been provided. Indeed, defining 
\begin{equation}
 \bar{V} (\vec z, \vecg\xi) := V \left( \vec z, (\vec I - \vec z \otimes \vec z) \vecg\xi \right),\quad (\vec z, \vecg\xi)\in\mS^{n-1}\times \Rnn, 
\label{prop:tqc}
\end{equation}
and $\bar{V}^{qc}$ to be the quasiconvexification of $\bar{V}$ with respect to the second variable, then $V^{tqc} = \bar{V}^{qc}|_{T^n \mS^{n-1}}$ (see \cite[Example~2.4]{DaFoMaTr99}).

With the following theorem, we obtain (a) the (sequential) weakly lower semicontinuity result in $W^{1,A}$ for the integral functional
\begin{equation*}
J(\vec n):= \int_{\Omega} V (\vn (\vec y), D \vn (\vec y)) \, \dd \vec y,
\end{equation*}
and then (b) the integral representation of the relaxed energy $J^*(\vec n)$.

\begin{thm}\label{th:tqc}
Let $\Omega \subset \R^n$ be open and bounded. Let $V : T^n \mS^{n-1}\to [0, \infty)$ be continuous and satisfy the growth condition
\begin{equation}
 V (\vec z, \vecg\xi) \leq C \left( 1 + A(\|\vecg\xi\|)\right) ,\qquad (\vec z, \vecg\xi) \in T^n \mS^{n-1},
\label{growcon}
\end{equation}
for some $C > 0$.
Let $\vn \in W^{1,A} (\Omega, \mS^{n-1})$.
The following hold:

\begin{enumerate}
\item[\emph{(a)}] 
If $V$ is tangentially quasiconvex then, for any sequence $\{ \vn_j \}_{j \in \N} \subset W^{1,A} (\Omega, \mS^{n-1})$ such that $\vn_j\weakc \vn$ in $W^{1,A} (\Omega, \mS^{n-1})$, we have
\[
 \int_{\Omega} V (\vn (\vec y), D \vn (\vec y)) \, \dd \vec y \leq \liminf_{j \to \infty} \int_{\Omega} V (\vn_j (\vec y), D \vn_j (\vec y)) \, \dd \vec y .
\]

\item[\emph{(b)}] 
\begin{equation}
\begin{split}
&\displaystyle\int_{\Omega} V^{tqc} (\vn (\vec y), D \vn (\vec y)) \, \dd \vec y\\
& =\inf_{\{\vn_j\}} \left\{ \liminf_{j \to \infty} \int_{\Omega} V (\vn_j (\vec y), D \vn_j (\vec y)) \, \dd \vec y :\quad \vn_j \weakc \vn \text{ in } W^{1,A} (\Omega, \mS^{n-1}) \right\}\,. 
\end{split}
\label{stepb}
\end{equation}
\end{enumerate}
\end{thm}

\proof
(a) The proof is \emph{verbatim} the same as in \cite[Proposition~2.5]{DaFoMaTr99}. The only difference is that, in the final estimate therein, the classical results of lower semicontinuity for quasi-convex integrands with standard $p$-growth are replaced by the corresponding generalization to the Orlicz setting \cite[Theorem~3.2]{FocMas}.\\
(b) The argument of \cite[Theorem~3.1]{DaFoMaTr99}, concerning with the relaxation of the analogous integral functional in $W^{1,p}$, applies to our case with minor modifications. We then omit the details of the proof, just mentioning the main steps. First, denoting by $\bar{J}(\vn)$ and $\mathcal{F}(\vn)$
the left and right hand sides of \eqref{stepb}, respectively, with (a) we immediately deduce that 
\begin{equation}
\mathcal{F}(\vn)\geq\bar{J}(\vn). 
\end{equation}
To prove the reverse inequality, we need to introduce an auxiliary localized version of $\mathcal{F}(\vn)$; namely,
\begin{equation}
\begin{split}
\mathcal{F}_\infty(\vn, U)= \inf_{\{\vn_j\}} \biggl\{ \liminf_{j \to \infty} &\int_{U} V (\vn_j (\vec y), D \vn_j (\vec y)) \, \dd \vec y : \quad\vn_j \weakc \vn \text{ in } W^{1,A} (\Omega, \mS^{n-1}),\\
 & \vn_j\to\vn \mbox{ uniformly, there exists a compact subset }\\
& \mbox{ $K\subset\mS^{n-1}$ s.t. $\vn_j(\vec y)=\vn(\vec y)$ if $\vec y\not\in K$} \biggr\}
\end{split}
\end{equation}
for every open subset $U\subset\Omega$. Notice that $\mathcal{F}_\infty(\vn, \Omega)\geq \mathcal{F}(\vn)$. Then, we have to show that $\mathcal{F}_\infty(\vn, \cdot)$ is the trace, on the class of open subsets of $\Omega$, of a finite Radon measure, absolutely continuous with respect to the $n$-dimensional Lebesgue measure $\mathcal{L}^n$ on $\R^n$. 

As for the subadditivity property
\begin{equation*}
\mathcal{F}_\infty(\vn, U)\leq \mathcal{F}_\infty(\vn, U') + \mathcal{F}_\infty(\vn, \overline{U}\backslash U''),
\end{equation*}
where $U, U', U''$ are open subsets of $\Omega$ such that $U''\ssubset U'\ssubset U$, the proof is not affected (up to minor modifications) by the Orlicz growth of the gradients.
Moreover, using the growth assumption \eqref{growcon}, we have
$$\mathcal{F}_\infty(\vn, U)\leq C \left( 1 + A(\|D\vn\|)\right)\mathcal{L}^n\res\Omega(U).$$
To conclude, we need the bound
\begin{equation}
\frac{\dd\mathcal{F}_\infty(\vn, \cdot)}{\dd\mathcal{L}^n}(\vec x_0)\leq V^{tqc} (\vn (\vec x_0), D \vn (\vec x_0)), \quad \mbox{for $\mathcal{L}^n$ a.e. $\vec x_0\in\Omega$},
\label{stimata2}
\end{equation}
where $\frac{\dd\mathcal{F}_\infty(\vn, \cdot)}{\dd\mathcal{L}^n}$ is the Radon Nikodym derivative of $\mathcal{F}_\infty(\vn, \cdot)$ with respect to $\mathcal{L}^n$. Indeed, \eqref{stimata2} is equivalent to $\mathcal{F}_\infty(\vn, \Omega)\leq \bar{J}(\vn)$, whence $\mathcal{F}(\vn)\leq \bar{J}(\vn)$.\\
\indent
We briefly sketch the proof of \eqref{stimata2}, referring the interested reader to \cite[pp. 201--206]{DaFoMaTr99} for details. Let ${\vec x_0}\in \Omega$ be a Lebesgue point for $\vn, D\vn$ such that $\frac{\dd\mathcal{F}_\infty(\vn, \Omega)}{\dd\mathcal{L}^n}(\vec x_0)$ exists and is finite, and set $\vec z_0=\vn(\vec x_0)$. Then, with fixed $\eta>0$, from Definition~\ref{defn:tangq}(b) there exists a test function $\vecg\varphi_\eta\in W^{1,\infty} (B_1, T_{\vec z_0}\mS^{n-1})$ such that
\begin{equation}
V^{tqc} (\vec n(\vec x_0),D\vn(\vec x_0))+\eta\geq\fint_{B_1} \! V(\vn(\vec x_0), D\vecg\varphi_\eta(\vec y))\, \dd \vec y.
\label{boundpart}
\end{equation}
Furthermore, a local argument based on the smooth projection of a neighborhood of $\vec z_0$ onto the sphere $\mS^{n-1}$, involving $\vn$, $\vecg\varphi_\eta$ and the continuity properties of $V$, provides a sequence $\vn_j$ converging uniformly to $\vn$ on $\Omega$ and with equiabsolutely integrable $\{D\vn_j\}$. The conclusion then follows by several estimates exploiting \eqref{boundpart} and the properties of $\{\vn_j\}$ and $\{D\vn_j\}$.

\endproof

Since finite-valued quasiconvex functions are continuous (because they are rank-one convex), we infer that any tangentially quasiconvex $V : T^n \mS^{n-1} \to [0, \infty)$ is continuous in the second variable.

\section{Compactness, lower semicontinuity and existence of minimizers for functionals defined in the deformed configuration}\label{complow}

In this section we prove existence of minimizers of $I$ on a suitable set, under the assumptions that $W$ is polyconvex in the first variable and $V$ is tangentially quasiconvex. 

We first introduce the \emph{admissible set} $\mathcal{B}$.
\begin{defn}
Let $\Omega \subset \R^n$ be a bounded Lipschitz domain, representing the reference configuration of the sample.
Let $\Gamma$ be an $(n-1)$-rectifiable subset of $\partial \Omega$, and let $\vec u_0:\Gamma\to \R^{n}$ be a given function.
We define the admissible set $\B$ as the set of pairs $(\vec u,\vn) \in L^1 (\Omega, \R^n) \times L^1 (\R^n, \R^n)$ such that $\vec u\in \A_{{\beta}}(\Omega)$, $\vec u|_{\Gamma}=\vec u_0$ in the sense of traces, $D \vec u (\vec x) \in \Rnn_+$ for a.e.\ $\vec x \in \Omega$,
\[
 \vn|_{\imT(\vec u,\Omega)} \in W^{1,A} (\imT (\vec u, \Omega) ,\mS^{n-1}) \quad \text{and} \quad \vn|_{\R^n \setminus \imT(\vec u,\Omega)} = 0 .
\]
\end{defn}

From the physical point of view, $\vec u$ can be meant to be the elastic deformation of the sample, with a given boundary condition $\vec u_0$, and $\vn$ to be the nematic director field evaluated in the deformed configuration of the sample with respect to $\vec u$.\\

We define the energy functionals
\begin{equation}\label{eq:Idefinitions}
 I, \, I_{\mec}, \, I_{\nem}: L^1 (\Omega, \R^n) \times L^1 (\R^n, \R^n) \to [0,+\infty]
\end{equation}
describing the nematic elastomer as follows:
\[
 I_{\mec} (\vec u, \vn) = \begin{cases}
 \displaystyle \int_{\Omega} W(D\vec u(\vec x),\vn(\vec u(\vec x)))\, \dd \vec x , & \text{if } (\vec u, \vn) \in \B , \\
 +\infty , & \text{otherwise,}
 \end{cases}
\]
\[
 I_{\nem} (\vec u, \vn) = \begin{cases}
 \displaystyle \int_{\imT(\vec u,\Omega)} V(\vn(\vec y), D\vn(\vec y)) \, \dd \vec y  , & \text{if } (\vec u, \vn) \in \B , \\
 +\infty , & \text{otherwise,}
 \end{cases}
\]
and, finally, 
\begin{equation}
\begin{split}
I(\vec u, \vn):&= I_{\mec}(\vec u, \vn) + I_{\nem}(\vec u, \vn)\\
                   &=\displaystyle \int_{\Omega} W(D\vec u(\vec x),\vn(\vec u(\vec x)))\, \dd \vec x +  \displaystyle \int_{\imT(\vec u,\Omega)} V(\vn(\vec y), D\vn(\vec y)) \, \dd \vec y
\end{split}
\end{equation}
if $(\vec u, \vn) \in \B$, $I(\vec u, \vn)=+\infty$ elsewhere.\\

The following result establishes the compactness for sequences bounded in energy and the lower semicontinuity of $I$ in $\B$ with respect to strong $L^1$-topology. The proof can be derived by combining the arguments of \cite[Proposition~4.1]{MCOl}, \cite[Theorem~8.2, Propositions~7.1 and 7.8]{HBMC}, with the necessary adaptation to the Orlicz-Sobolev setting.

\begin{prop}\label{prop:lowersemicont}
Let $W:\R^{n\times n}_{+}\times \mS^{n-1}\to [0,\infty)$ be continuous, polyconvex and such that
\begin{equation}\label{eq:coercWtilde}
 W (\vec F, \vn) \geq c \, A(\|\vec F\|) + h (\det \vec F) , \qquad \vec F \in \Rnn_+ , \quad \vn \in \mS^{n-1}
\end{equation}
for a constant $c>0$ and a Borel function $h : (0, \infty) \to [0, \infty)$ with
\begin{equation}\label{eq:h1}
 \lim_{t \searrow 0} h (t) = \lim_{t \to \infty} \frac{h (t)}{t} = \infty.
\end{equation}
Let $V : T^n \mS^{n-1} \to [0, \infty)$ be continuous and tangentially quasiconvex such that
\begin{equation}\label{eq:Vsgrowth}
 c A(\| \vecg\xi \|) - \frac{1}{c} \leq V (\vec z, \vecg\xi) \leq \frac{1}{c} \left( 1 + A(\| \vecg\xi \|) \right) ,\qquad (\vec z, \vecg\xi) \in T^n \mS^{n-1} .
\end{equation}
Then
\begin{description}
\item[(i)] \emph{(compactness)} for every sequence $\{(\vec u_j, \vn_j)\}_{j\in\N} \subset L^1 (\Omega, \R^n) \times L^1 (\R^n, \R^n)$ such that $\sup_{j \in \N} I(\vec u_j, \vn_j) < \infty$, there exist a subsequence (not relabelled) and $(\vec u, \vn) \in \B$ such that
\begin{equation}\label{eq:ujnjun}
 (\vec u_j, \vn_j) \to (\vec u, \vn) \text{ in } L^1 (\Omega, \R^n) \times L^1 (\R^n, \R^n) \qquad \text{as } j\to\infty;
\end{equation}
\item[(ii)] \emph{(lower semicontinuity)} for every sequence $\{(\vec u_j, \vn_j)\}_{j\in\N}$ such that \eqref{eq:ujnjun} holds, we have
\begin{equation}\label{eq:Isemicont}
 I (\vec u, \vn) \leq \liminf_{j \to \infty} I (\vec u_j, \vn_j) .
\end{equation}
\end{description}
\end{prop}

\begin{proof}[Sketch of the proof]
We can assume, eventually passing to a subsequence, that the ${\lim\inf}$ of the right-hand side of \eqref{eq:Isemicont} is a limit, and that it is finite. In particular, $(\vec u_j, \vec n_j)\subset \mathcal{B}$, so that ${\vec u_j}_{|\Gamma}=\vec u_0$ for every $j$. We first notice that, by assumptions \eqref{eq:coercWtilde}-\eqref{eq:Vsgrowth} and de la Vall\'ee Poussin's criterion (Theorem~\ref{lavpous}), the sequence $\{D\vec u_j\}$ is bounded in $L^A(\Omega,\Rnn)$ and $\{\det D\vec u_j\}$ is equiintegrable. By Poincar\`e inequality and the boundary condition, also $\{\vec u_j\}$ is bounded in $W^{1,A}(\Omega,\R^n)$. Therefore, as showed in the proof of \cite[Theorem~8.2]{HBMC}, there exists $\vec u\in\A_{{\beta}}(\Omega)$ with $\vec u_{|\Gamma}=\vec u_0$ such that $\vec u_j \weakc \vec u \text{ in } W^{1,A} (\Omega, \R^n)$ and $\det D \vec u_j \weakc \det D \vec u \text{ in } L^1 (\Omega)$. 

Now, by the area formula and the fact that $\imT (\vec u_j, \Omega)=\imG (\vec u_j, \Omega)$, the boundedness in $L^1$ of $\{\det D\vec u_j\}$ implies that of $\|\vn_j\|_{L^A(\imT (\vec u_j, \Omega),\R^n)}$. Therefore, \cite[Proposition~7.1]{HBMC} yields the existence of $\vn\in W^{1,A}(\imT (\vec u, \Omega),\mS^{n-1})$ such that $(\vec u, \vec n)\in \mathcal{B}$ and, along a subsequence, 
$\chi_{\imT (\vec u_j, \Omega)} \vn_j \to \chi_{\imT (\vec u, \Omega)}\vn$ in $L^A (\R^n, \R^n)$ and a.e.,  $\chi_{\imT (\vec u_j, \Omega)} D \vn_j \weakc \chi_{\imT (\vec u, \Omega)} D \vn$ in $L^A (\R^n, \Rnn)$ as  $j \to \infty$, where $\chi_{\imT (\vec u, \Omega)} D \vn$ denotes the extension of $D \vn$ by zero outside $\imT (\vec u, \Omega)$, and analogously for $\chi_{\imT (\vec u_j, \Omega)} D \vn_j$. Moreover, by \cite[Proposition~7.8]{HBMC}, we have
\[
\displaystyle \int_{\Omega} W(D\vec u(\vec x),\vn(\vec u(\vec x)))\, \dd \vec x \leq \liminf_{j \to \infty} \displaystyle \int_{\Omega} W(D\vec u_j(\vec x),\vn_j(\vec u_j(\vec x)))\, \dd \vec x.
\]

As for the lower semicontinunity of the nematic term of the energy, we follow the proof of \cite[Proposition~4.1]{MCOl}. Let $K \ssubset \imT (\vec u, \Omega)$ be open. Then, by Proposition~\ref{th:ujINV}(i) applied to the compact set $\bar{K}$, there exists $j_K \in \N$ such that for all $j \geq j_K$ we have $K\subset \imT (\vec u_j, \Omega)$.
Therefore, $\vn_j \weakc \vn$ in $W^{1,A} (K, \R^n)$ as $j \to \infty$ and by Theorem \ref{th:tqc} we get
\[
\begin{split}
 \int_{K} V (\vn (\vec y), D \vn (\vec y)) \, \dd \vec y &\leq \liminf_{j \to \infty} \int_{K} V (\vn_j (\vec y), D \vn_j (\vec y)) \, \dd \vec y\\ &\leq \liminf_{j \to \infty } \int_{\imT (\vec u_j, \Omega)} V (\vn_j (\vec y), D \vn_j (\vec y)) \, \dd \vec y.
\end{split}
\]
The result then follows by the arbitrariness of $K$.
\end{proof}

Once the compactness and lower semicontinuity properties have been proved with Proposition \ref{prop:lowersemicont}, the direct method of the calculus of variations (if necessary, see, e.g.,~\cite{Daco}) yields the following result on the existence of minimizers of $I$ on the admissible set $\mathcal{B}$ (compare with \cite[Thm.~5.1]{HS}).

\begin{thm}\label{th:existence}
Let $W:\R^{n\times n}_{+}\times \mS^{n-1}\to [0,\infty)$ be continuous, polyconvex and such that \eqref{eq:coercWtilde}-\eqref{eq:h1} hold for a constant $c>0$ and a Borel function $h : (0, \infty) \to [0, \infty)$.
Let $V : T^n \mS^{n-1} \to [0, \infty)$ be continuous and tangentially quasiconvex such that \eqref{eq:Vsgrowth} holds.
If $\B\neq \varnothing$ and $I$ is not identically infinity, then $I$ attains its minimum in $\B$.
\end{thm}

\section{Construction of a recovery sequence and relaxation}\label{sec:recoveryrelax}

The main aim of this section is the construction of a recovery sequence providing the upper bound inequality that, combined with the compactness and lower semicontinuity results obtained in the previous section, will allow us to obtain the relaxation theorem. 

In order to do that, we list below the coercivity, growth and continuity assumptions on the energy functions $W$ and $V$.\\
\\
{\bf (a) Assumptions on $W$:}
\\
\begin{description}
\item[($W_1$)] $W: \R^{n\times n}_{+} \times \mS^{n-1} \to [0,\infty)$ is continuous;\\
\item[($W_2$)] there exist a convex $h:(0,\infty)\to [0,\infty)$ and constants $c_h,c_W>0$ such that 
\begin{align}
 & h(st) \leq c_h \left( 1 + h(s) \right) \left( 1+h(t) \right) , \qquad s, t > 0 ,\label{condh} \\
 & \lim_{t\to\infty}\frac{h(t)}{t}=\infty , \qquad \liminf_{t\to 0}\frac{h(t)}{{\log^\beta(e+\frac{1}{t})t^{-\frac{(p-1)(n+1)}{n-p+1}}}}>0,\label{2condh}
 \end{align}
 and for all $\vec F \in \Rnn_+$ and $\vec n\in \mS^{n-1}$,
\begin{equation}
 \frac{1}{c_W} \left( A(\| \vec F \|) + h(\det \vec F) \right) -c_W \leq W(\vec F, \vec n) \leq c_W \left(A(\| \vec F \|) + h(\det \vec F) + 1 \right).
\label{condW} 
\end{equation}
\item[($W_3$)] there exists a bounded Borel $\omega:[0,2] \to [0,\infty)$, with $\lim_{t\to0}\omega(t)=0$, such that
\begin{equation}
\left| W(\vec F,\vec n)-W(\vec F,\vec m) \right|\leq \omega \left( \|\vec n-\vec m\| \right) W(\vec F,\vec n)
\end{equation}
for all $\vec F \in \Rnn_+$ and $\vec n,\vec m\in \mS^{n-1}$.
\end{description}

The function $W$ is extended by infinity to $(\Rnn \setminus \R^{n\times n}_{+}) \times \mS^{n-1}$.
Observe that if $(\vec u, \vec n) \in \B$ satisfies $I_{\mec} (\vec u, \vec n) < +\infty$ then $\vec u \in \A_{{\beta}}(\Omega)$.\\
\\
{\bf (b) Assumptions on $V$:}\\
\\
{\bf ($V_1$)} $V : T^n \mS^{n-1} \to [0, \infty)$ is continuous and there exists $c_V>0$ such that
\begin{equation}
\frac{1}{c_V}A(\|\vecg\xi\|)-c_V\leq V(\vec z, \vecg\xi)\leq c_V (A(\|\vecg\xi\|)+1) , \qquad (\vec z, \vecg\xi) \in T^n \mS^{n-1}.
\end{equation}

We define the admissible set $\mathcal{B}$ as in Section~\ref{complow} and the functionals $ I, \, I_{\mec}, \, I_{\nem}$ as in \eqref{eq:Idefinitions}. Correspondingly, we introduce the relaxed functionals $ I^{*}, \, I^{*}_{\nem} , \, I^{*}_{\mec}$, defined on $L^1 (\Omega, \R^n) \times L^1 (\R^n, \R^n)$ as 
\[
 I^{*}_{\mec} (\vec u, \vn) = \begin{cases}
 \displaystyle \int_{\Omega} W^{qc}(D\vec u(\vec x),\vn(\vec u(\vec x)))\, \dd \vec x , & \text{if } (\vec u, \vn) \in \B , \\
 +\infty , & \text{otherwise,}
 \end{cases}
\]
\[
 I^{*}_{\nem} (\vec u, \vn) = \begin{cases}
 \displaystyle \int_{\imT(\vec u,\Omega)} V^{tqc}(\vn(\vec y), D\vn(\vec y)) \, \dd \vec y  , & \text{if } (\vec u, \vn) \in \B , \\
 +\infty , & \text{otherwise,}
 \end{cases}
\]
and $I^*:= I^*_{\mec} + I^*_{\nem}$, where $W^{qc}$ is the quasiconvexification of $W$ with respect to the first variable and $V^{tqc}$ is the tangential quasiconvexification of $V$.\\

The main result of this section is the following theorem.

\begin{thm}\label{recseq}
Let $\Omega \subset \R^{n}$ be open, bounded and with Lipschitz boundary. Assume that $W$ and $V$ comply with $(W_1)-(W_3)$ and $(V_1)$, respectively. Then, for any $(\vec u, \vn) \in L^1 (\Omega, \R^n) \times L^1 (\R^n, \R^n)$ there is a sequence $\{({\vec u}_j, \vn_j)\}_{j\in\N}\subset L^1 (\Omega, \R^n) \times L^1 (\R^n, \R^n)$ such that
\[
 ({\vec u}_j, \vn_j) \to (\vec u, \vn) \text{ in } L^1 (\Omega, \R^n) \times L^1 (\R^n, \R^n) \qquad \text{as } j\to\infty
\]
and
\begin{equation}
 \limsup_{j\to\infty} I ({\vec u}_j, \vn_j) \leq I^* (\vec u, \vn) .
\label{eqn:limsup}
\end{equation}
\end{thm}

The proof of this result relies on the following Lemma~\ref{approxubyz}. As showed by Lemma~\ref{lem3.10}, the local composition of the limiting deformation $\vec u$ with a Lipschitz map $\vecg\rho$, satisfying the correct integrability condition, is still an admissible deformation. What we need to know is that the unrelaxed mechanical energy of such composition is near the relaxed mechanical energy of $\vec u$.

We will exploit the following technical result \cite[Lemma 3.1]{CoDo15} dealing with the integrability of the product of suitable translations of $L^1$ functions.
\begin{lem}\label{lem composition L1}
Let $\vec x_0\in\R^n$ and $r>0$.
Let $\vecg\psi\in W^{1,\infty}(B(\vec 0,r),\overline{B(\vec 0,r)})$, $g\in L^{1}(B(\vec 0,r))$ and $f\in L^{1}(B(\vec x_0,2r))$.
Then, there exists a measurable set $E\subset B(\vec x_0,r)$ of positive measure such that for any $\vec a_0\in E$, the function
\[ \tilde{f} (\vec x) :=f(\vec a_0 + \vecg\psi(\vec x - \vec a_0)) \, g(\vec x - \vec a_0) , \qquad \vec x \in B(\vec a_0,r) \]
belongs to $L^{1}(B(\vec a_0,r))$ and 
\[\| \tilde{f}\|_{L^{1}(B(\vec a_0,r))}\le \frac{1}{|B(\vec 0,r)|}\|f\|_{L^{1}(B(\vec x_0,2r))}\|g\|_{L^{1}(B(\vec 0,r))}.\]
\end{lem}
We are now in position to state and prove the first key result of this section.

\begin{lem}\label{approxubyz} Assume that $W: \R_{+}^{n\times n}\longrightarrow[0,\infty)$ satisfies $(W_1)-(W_3)$ and let $\vec F\in\R_{+}^{n\times n}$, $\vec m\in\mS^{n-1}$ and $\eta\in (0,1)$ be fixed. Then there exists $\delta>0$ such that, for any ball $B(\vec x_0,r)$, any $\vn\in L^{\infty} (\imT(\vec u,B(\vec x_0,r)),\mS^{n-1})$ and any $\vec u\in\mathcal{A}_{{\beta}}(\Omega)$ complying with
\begin{equation}\label{eq: approxledelta}
\fint_{B(\vec x_0,r)} \bigl(A(\|D\vec u - \vec F\|)+|h(\det D\vec u)-h(\det \vec F)|+A(\|\vec n\circ \vec u-\vec m\|)\bigr) \dd \vec x\leq \delta,
\end{equation}
there exist $\vec a_0\in B\left(\vec x_0,\frac{r}{2}\right)$ and $\vec z\in\A_{{\beta}}(B(\vec x_0,r))$ with $\vec z = \vec u$ in $B(\vec x_0,r)\setminus B\left(\vec a_0,\frac{r}{2}\right)$, $\imT (\vec z, \Omega) = \imT (\vec u, \Omega)$,
\begin{equation}\label{eq: lemconv Wqc}
\int_{B\left(\vec a_0,\frac{r}{2}\right)}W(D\vec z, \vn\circ \vec z)\, \dd \vec x\leq \int_{B\left(\vec a_0,\frac{r}{2}\right)}(W^{qc}(D\vec u, \vn\circ \vec u)+C_1\eta)\, \dd \vec x
\end{equation}
and
\begin{equation}\label{eq: lemconvlpbound}
\int_{B(\vec x_0,r)}A(\|\vec u - \vec z\|) \, \dd \vec x \leq C_2\int_{B(\vec x_0,r)}(W^{qc}(D\vec u,\vn\circ \vec u)+1)\, \dd \vec x
\end{equation}
for some positive constants $C_1=C_1(W)$ and $C_2=C_2(r,A,W)$.
\end{lem}

\proof
The proof is an adaptation to the Orlicz-Sobolev setting of the arguments developed in \cite[Lemma~6.2]{MCOl} and \cite[Lemma~3.2]{CoDo15} for the Sobolev spaces $W^{1,p}$.

The estimate \eqref{eq: lemconvlpbound} follows from Poincar\'{e}'s inequality \eqref{poinc}, the growth condition $(W_2)$, the convexity, monotonicity and $\Delta_2$ property of $A$ \eqref{delta2} with constant $\mu$ and \eqref{eq: lemconv Wqc}. Indeed, we have
\begin{align*}
 & \int_{B(\vec x_0,r)}A(\|\vec u - \vec z\|) \, \dd \vec x = \int_{B (\vec a_0, \frac{r}{2})} A(\|\vec u - \vec z\|) \, \dd \vec x \leq C(r,\mu)\int_{B (\vec a_0, \frac{r}{2})} A(\|D\vec u - D\vec z\|) \, \dd \vec x \\
 & \leq C(r,\mu) \int_{B (\vec a_0, \frac{r}{2})} A(\|D\vec u\| + \|D\vec z\|) \dd \vec x \leq C(r,\mu) \frac{\mu}{2}\int_{B (\vec a_0, \frac{r}{2})} A(\|D\vec u\|) + A(\|D\vec z\|) \dd \vec x\\
&\leq C'(r,\mu,W) \int_{B (\vec a_0, \frac{r}{2})} \left( W^{qc} (D\vec u, \vn \circ \vec u) + W (D\vec z, \vn \circ \vec z) + 1 \right) \dd \vec x \\
 & \leq C_2 \int_{B (\vec a_0, \frac{r}{2})} \left( W^{qc} (D\vec u, \vn \circ \vec u) + 1 \right) \dd \vec x \leq C_2 \int_{B(\vec x_0,r)} \left( W^{qc} (D\vec u, \vn \circ \vec u) + 1 \right) \dd \vec x ,
\end{align*}
where $C_2>0$ is a constant depending on $r,A$ (through $\mu$) and $W$.

We then prove \eqref{eq: lemconv Wqc}. With fixed $\vec m\in \mS^{n-1}$, by Definition \ref{de:quasiconvexification} of the quasiconvexification of $W$, corresponding to $\eta$ there exists $\vecg\varphi_{\eta}\in W^{1,\infty} (B_{\frac{r}{2}}, \R^n )$ such that $\vecg\varphi_{\eta}(\vec x)=\vec F \vec x$ on $\partial B_{\frac{r}{2}}$, $\det D\vecg\varphi_{\eta}>0$ a.e.\ and
\begin{equation}\label{eq: approxqcvp}
\fint_{B_{\frac{r}{2}}} W(D\vecg\varphi_{\eta}(\vec x),\vec m)\, \dd x\leq W^{qc}(\vec F,\vec m)+\eta.
\end{equation}

The function $\vecg\rho:=\vec F^{-1}\vecg\varphi_\eta$ is Lipschitz and $\vecg\rho=\vec {id}$ on $\partial B_{\frac{r}{2}}$, so that by degree theory $\vecg\rho(B_{\frac{r}{2}})\subset \overline{B_{\frac{r}{2}}}$.
 Moreover, it is well known (see, e.g. \cite[Th.\ 8]{Sverak88}) that $\vecg\rho$ is invertible and that $\vecg\rho^{-1}\in W^{1,1}(B_{\frac{r}{2}})$. For the sake of brevity, we set $B:=B(\vec x_0,r)$.

For $\vec a_0\in B (\vec x_0,\frac{r}{2} )$ that will be suitably chosen later, we consider $B':=B (\vec a_0,\frac{r}{2})$, set $\rhotilde(\vec x)=\vecg\rho(\vec x-\vec a_0)+\vec a_0$ and 
\[\vec z(\vec x):=\begin{cases} \vec u \circ \tilde{\vecg\rho}(\vec x)= \vec u(\vec F^{-1}\vecg\varphi_\eta(\vec x-\vec a_0)+\vec a_0) &\text{in } B' , \\
\vec u(\vec x)&\text{in } B\setminus B' .
\end{cases}\]
We then have $\vec z = \vec u$ in $B\setminus B'$, $\imT(\tilde{\vecg\rho},B')=B'$ and $\tilde{\vecg\rho}^{-1}\in W^{1,1}(B',\R^{n})$. 

Since $\vec u\in W^{1,1}_{\mbox{loc}}(\Omega,\R^n)$, by \cite[Lemma~5.2]{MCOl} and \cite[Lemma A.2]{CoDo15}, there exists a $\mathcal{L}^n$-null set $N$ such that for all $\vec a_0\in B(\vec x_0, \frac{r}{2}) \setminus N$ we have that $\vec z\in W^{1,1}(B',\R^{n})$, $\det D\vec z\in L^{1}(B)$. Moreover, since $\tilde{\vecg\rho}|_{\partial B'}=\vec {id}|_{\partial B'}$  we have $\vec u\circ \tilde{\vecg\rho}|_{\partial B'}=\vec u|_{\partial B'}$ and then $\vec z\in W^{1,1}(B,\R^{n})$. Choose $E$ and $\vec a_0\in E\setminus N$ using Lemma~\ref{lem composition L1} applied to $B'$ with $\vecg\psi=\vec F^{-1}\vecg\varphi_\eta$, $f=A(\|D\vec u - \vec F\|)+|h(\det D\vec u)-h(\det \vec F)|$ and $g=1+h(\det(\vec F^{-1}D\vecg\varphi_\eta))$. Then, by \eqref{eq: approxledelta}, we obtain
\begin{equation}\label{eq: approxqcledelta}
\begin{split}
 &\fint_{B'}(1+h(\det D\rhotilde))\bigl(A(\|D\vec u - \vec F\|)+|h(\det D\vec u)-h(\det \vec F)|\bigr)\circ \rhotilde\, \dd \vec x\\
&\leq \frac{|B(\vec x_0,r)|}{|B'|}\fint_{B(\vec x_0,r)}\bigl(A(\|D\vec u - \vec F\|)+|h(\det D\vec u)-h(\det \vec F)|\bigr)\, \dd \vec x\fint_{B_{\frac{r}{2}}}(1+h(\det D\vecg\rho))\, \dd \vec x\\
&\le 2^n\delta\fint_{B_{\frac{r}{2}}}(1+h(\det D\vecg\rho))\, \dd \vec x=:c_\eta\delta,
\end{split}
\end{equation}
with $c_\eta$ depending on $\eta$ and $\vec F$.

With \eqref{eq: approxqcvp}, \eqref{condW} implies $h(\det D\vecg\varphi_\eta)\in L^{1} (B_{\frac{r}{2}})$ whence, in view of \eqref{2condh}, we also obtain ${\log^\beta\left(e+\frac{1}{\det D\vecg\varphi_\eta}\right)(\det D\vecg\varphi_\eta)^{-\frac{(p-1)(n+1)}{n-p+1}}\in L^{1} (B_{\frac{r}{2}})}$. This ensures that $\rhotilde$ complies with the integrability assumption \eqref{conddet} of Lemma~\ref{lem3.10} on $B'$. Furthermore, from \eqref{condh} we deduce that $h(\det D\vecg\rho)\in L^{1} (B_{\frac{r}{2}})$. 
Therefore, since $\det D\vecg\varphi_\eta>0$ a.e., from the absolute continuity of the integral there exists a constant $\gamma=\gamma(\vec F, \vec m, \eta)>0$ such that  
\begin{equation}\label{eq: approxqcdet<gam1}
\int_{B_{\frac{r}{2}}\cap \{\det D\vecg\varphi_\eta<\gamma\}}\bigl(1+h(\det D\vecg\rho)\bigr)\, \dd \vec x\leq \frac{|B_{\frac{r}{2}}|\eta}{\left(3+\frac{\mu}{2}c_{\eta,\mu}\right)\left(A(1)+A(\|\vec F\|)+h(\det \vec F)\right)},
\end{equation}
where $c_{\eta,\mu}$ is the smallest constant such that $A(\|D\vecg\rho\|_{L^\infty}t)\leq c_{\eta,\mu}A(t)$, and
\begin{equation}\label{eq: approxqcdet<gam2}
\int_{B_{\frac{r}{2}}\cap \{\det D\vecg\varphi_\eta<\gamma\}}\left(1+A(\|D\vecg\varphi_\eta\|)+h(\det D\vecg\varphi_\eta)\right)\, \dd \vec x\le \frac{1}{c_W}|B_{\frac{r}{2}}|\eta,
\end{equation}
where $c_W$ is the constant of \eqref{condW}.

Let  $R_\eta=\|D\rhotilde\|_{L^{\infty}}$ and $M_\eta=\|D\vecg\varphi_\eta\|_{L^{\infty}}$.
 Since by assumption $(W_1)$, $W$ is continuous in $\R^{n\times n}_{+}\times\mS^{n-1}$, there exists $\varepsilon>0$ not depending on $\vec u$, $\vn$ or $\delta$ with $\varepsilon R_\eta \leq 1$ and $\varepsilon \leq 1$ such that
\begin{equation}\label{eq: approxconvqcWcont}
|W(\vecg\sigma,\vecg{\ell})-W(\vecg\zeta,\vec{k})|\leq \eta
\end{equation}
for all $(\vecg\sigma,\vecg \ell)$, $(\vecg\zeta, \vec{k})\in \R^{n\times n}_{+}\times \mS^{n-1}$ complying with $\|\vecg\zeta\|\le M_\eta$, $\det \vecg\zeta\geq\gamma$ and $\|\vecg\sigma-\vecg\zeta\|+\|\vecg{\ell}-\vec{k}\|\leq R_{\eta}$.
In addition, by Proposition \ref{pr:Wqccont} and the continuity of $h$, the number $\varepsilon$ can be chosen in order to obtain
\begin{equation}\label{eq: approxconvqcWqccont}
|W^{qc}(\vecg\zeta,\vecg{\ell})-W^{qc}(\vec F,\vec m)|+|h(\det \vecg\zeta)-h(\det \vec F)|\leq \eta
\end{equation}
for all $\vecg\zeta \in \Rnn_+$ and $\vecg{\ell} \in \mS^{n-1}$ satisfying $\|\vecg \zeta - \vec F\|+\|\vecg{\ell}-\vec m\|\leq \varepsilon$.

We set $\widetilde{\vecg\varphi}_\eta(\vec x)=\vecg\varphi_\eta(\vec x - \vec a_0)$, and write
\[\int_{B'}\left(W(D\vec z,\vn\circ \vec z)-W^{qc}(D\vec u,\vn\circ \vec u)\right) \dd \vec x \] 
as the sum of the four integrals 
\begin{align*}
 J_1&=\int_{B'}\left(W(D\vec z,\vn\circ \vec z)-W(D\hvpe,\vn\circ \vec z)\right) \dd \vec x,\\
J_2&=\int_{B'}\left(W(D\hvpe,\vn\circ \vec z)-W(D\hvpe,\vec m)\right) \dd \vec x, \\
 J_3&=\int_{B'}\left(W(D\hvpe,\vec m)-W^{qc}(\vec F,\vec m)\right) \dd \vec x, \\
J_4&=\int_{B'}\left(W^{qc}(\vec F,\vec m)-W^{qc}(D\vec u,\vn\circ \vec u)\right) \dd \vec x,
\end{align*}
which will be estimated separately. From \eqref{eq: approxqcvp} we immediately deduce that 
\begin{equation}
J_3=\int_{B'}\left(W(D\hvpe(\vec x),\vec m)-W^{qc}(\vec F,\vec m)\right) \dd \vec x \leq \eta|B'|.
\label{stima3}
\end{equation}
To estimate $J_4$, we first define the set
\begin{equation*}
S_\epsilon:=\left\{\vec x\in B':\, \|D\vec u(\vec x)-\vec F\|+\|\vec m-\vn\circ \vec u(\vec x)\|\leq \varepsilon\right\}
\end{equation*}
and then we use \eqref{eq: approxconvqcWqccont} to get 
\begin{equation}
W^{qc}(\vec F,\vec m)\leq W^{qc}(D\vec u(\vec x),\vn\circ \vec u(\vec x))+\eta\quad\text{ for every $\vec x\in S_\epsilon$.}
\label{stimSeps}
\end{equation}
In $B'\backslash S_\epsilon$, instead, we use the Chebyshev's inequality \eqref{cheb}, the $\Delta_2$-condition \eqref{delta2} and \eqref{eq: approxledelta} to obtain, with \eqref{stimSeps},
\begin{align*}
J_4&\leq \eta \, | B'| + W^{qc}(\vec F,\vec m) |B'\backslash S_\epsilon|\\
&\le \eta \, |B'| + W^{qc}(\vec F,\vec m)\frac{\mu}{2A(\epsilon)}\int_{B'} \left(A(\|D\vec u(\vec x) - \vec F\|)+A(\|\vec n\circ \vec u(\vec x)-\vec m\|)\right) \dd \vec x\\
&\le \eta \, |B'| + W^{qc}(\vec F,\vec m)\frac{\mu}{2A(\epsilon)}|B|\delta.
\end{align*}
Thus,
\begin{equation}
J_4=\int_{B'}\left(W^{qc}(\vec F,\vec m)-W^{qc}(D\vec u(\vec x),\vn\circ \vec u(\vec x))\right) \dd \vec x \le 
(\eta+c_4\delta)|B'|,
\label{stima4}
\end{equation}
where $c_4=c_4(n,A,W,\vec F, \vn):=W^{qc}(\vec F,\vec m)\frac{2^{n-1}\mu}{A(\epsilon)}$. 

For what concerns the estimate of $J_2$, we consider the sets 
\[U'_\epsilon=\{\vec x\in B': \|\vn\circ \vec u(\vec x)-\vec m\| \geq \varepsilon R_\eta \} \quad \text{and} \quad U_\gamma=\{\vec x\in B': \det D\hvpe (\vec x)\geq \gamma\},\]
where $\varepsilon$ and $\gamma$ are the same as of \eqref{eq: approxconvqcWcont}.
With the change of variables $\vec x=\rhotilde^{-1}(\vec x')$, we can rewrite $J_2$ as
\begin{align*}
J_2=\int_{B'}\left(W((D\hvpe)\circ \rhotilde^{-1}(\vec x'),\vn\circ \vec u(\vec x'))-W((D\hvpe)\circ \rhotilde^{-1}(\vec x'),\vec m)\right)\det D \rhotilde^{-1}(\vec x')\, \dd \vec x',
\end{align*}
and it holds that
\begin{equation}
\int_{B'}\det D\rhotilde^{-1}(\vec x')\, \dd \vec x'= |B'|.
\label{intdet}
\end{equation}
We use \eqref{eq: approxconvqcWcont} and \eqref{intdet} on the set $\rhotilde(U_\gamma)\setminus U'_\epsilon$, and \eqref{condW} on $B'\setminus  (\rhotilde(U_\gamma)\setminus U'_\epsilon)$ to get
\begin{align*}
J_2&\le \int_{\vec \rhotilde(U_\gamma)\setminus U'_\epsilon}\eta \det D\rhotilde^{-1}(\vec x')\, \dd \vec x'+\int_{B'\setminus  (\rhotilde(U_\gamma)\setminus U'_\epsilon)}W((D\hvpe)\circ \rhotilde^{-1}(\vec x'),\vn\circ \vec u(\vec x'))\det D\rhotilde^{-1}(\vec x')\, \dd \vec x'\\
&\leq \eta \, |B'|+c_W\int_{B'\setminus (\rhotilde(U_\gamma)\setminus U'_\epsilon)}\left(1+ A(\|(D\hvpe)\circ \rhotilde^{-1}(\vec x')\|)+h(\det D\hvpe)\circ \rhotilde^{-1}(\vec x')\right)\det D\rhotilde^{-1}(\vec x')\, \dd \vec x'.
\end{align*}
Now, the change of variables $\vec x = \rhotilde^{-1}(\vec x')$ and \eqref{eq: approxqcdet<gam2} give
\begin{align*}
&c_W\int_{B'\setminus \rhotilde(U_\gamma)}\left(1+ A(\|(D\hvpe)\circ \rhotilde^{-1}(\vec x')\|)+h(\det D\hvpe)\circ \rhotilde^{-1}(\vec x')\right)\det D\rhotilde^{-1}(\vec x')\, \dd \vec x'\\
&= c_W\int_{B'\setminus U_\gamma}\left(1+ A(\|D\hvpe(\vec x)\|)+h(\det D\hvpe(\vec x))\right) \dd \vec x\leq \eta \, |B'|.
\end{align*}
 On the other hand, since for every $\vec x\in U_\gamma$ it holds that $\det D\rhotilde(\vec x)\geq \gamma\det \vec F^{-1}$, we infer that $\det D\rhotilde^{-1}\in L^{\infty}(\rhotilde(U_\gamma))$. Then, as a consequence of \eqref{eq: approxledelta}, $h(\det D\hvpe)\in L^{\infty}(U_\gamma)$ and the Chebyshev's inequality we get
\begin{align*}
&c_W\int_{U'_\epsilon\cap \rhotilde(U_\gamma)}\left(1+ A(\|(D\hvpe)\circ \rhotilde^{-1}(\vec x')\|)+h(\det (D\hvpe))\circ \rhotilde^{-1}(\vec x')\right)\det D \rhotilde^{-1}(\vec x')\, \dd \vec x'\\
&\leq \tilde{c} |U'_\epsilon|\leq \tilde{c}\frac{1}{A(\varepsilon)R_\eta} \, \delta \, |B'|,
\end{align*} 
with the constant $\tilde{c}$ depends only on $W$, $\gamma$ and $\eta$.

Therefore, we conclude that there exists a constant $\tilde{c}$ depending on $\eta$ and $W$ but not on $\delta$ such that
\begin{equation}
\begin{split}
J_2=\int_{B'}\left(W(D\hvpe,\vn\circ \vec z)-W(D\hvpe,\vec m)\right) \dd \vec x&\le \left(2\eta+\tilde{c}\frac{1}{A(\varepsilon)R_\eta}\delta\right) \, |B'|\\
&=: (2\eta + c_2\delta)|B'|,
\end{split}
\label{stima2}
\end{equation}
where $c_2=c_2(\eta,A,W)$. 

Now, we are left to estimate the integral $J_1$. For this, we introduce the set
\[U''_\epsilon=\{\vec x\in B': \|D\vec u(\vec x)-\vec F\|\circ \rhotilde\geq \varepsilon \}.\]
We note that, for every $\vec x\in B'$,
\[D\vec z(\vec x)=(D\vec u\circ \rhotilde(\vec x))D\rhotilde(\vec x)=\left[(D\vec u-\vec F)\circ \rhotilde(\vec x)\right]D\rhotilde(\vec x)+D\hvpe(\vec x)\] 
and that in $U_\gamma\setminus U''_\epsilon$ we have $\det D\hvpe\geq \gamma$ and $ \|D\vec u(\vec x)-\vec F\|\circ \rhotilde \leq \varepsilon$. We then get
\[\|D\vec z-D\hvpe\|\leq \left[\|D\vec u-\vec F\|\circ \rhotilde\right]\|D\rhotilde\|\leq \varepsilon R_{\eta} ,\]
whence, combined with \eqref{eq: approxconvqcWcont}, we infer that
\[\int_{U_\gamma\setminus U''_\epsilon}\left(W(D\vec z,\vn\circ \vec z)-W(D\hvpe,\vn\circ \vec z)\right) \dd \vec x\leq \eta|B'|.\]
Using the growth estimate \eqref{condW} we obtain
\[W(D\vec z,\vn\circ \vec z)\leq c_W\left(1+A([\|D\vec u\|\circ \rhotilde]\|D\rhotilde\|)+h((\det D\vec u)\circ \rhotilde\det D\rhotilde)\right).\]
Now, with $\|D\rhotilde\|\leq R_\eta$, the monotonicity of $A$ and \eqref{condh} we get that, in $B'$,
\begin{equation}\label{eq: approxboundWdz}
W(D\vec z,\vn\circ \vec z)\leq c_W\bigl[1+A(R_\eta \|D\vec u\|)\circ \rhotilde+(1+h(\det D\vec u)\circ \rhotilde)\bigr](1+h(\det D\rhotilde)).
\end{equation}
To estimate $J_1$ in $U''_\epsilon$, we note that from $\|D\vec u-\vec F\|\circ \rhotilde\geq \varepsilon $, the triangle inequality and the properties of $A$ we deduce that
\begin{equation}
\begin{split}
A(R_\eta\|D\vec u\|)\circ \rhotilde+2&\leq \frac{\mu}{2}\left(A(R_\eta\|D\vec u-\vec F\|)\circ \rhotilde+A(R_\eta\|\vec F\|)+\frac{4}{\mu}\right)\\
&= \frac{\mu}{2}\left(\frac{\mu A(R_\eta\|\vec F\|)+4}{\mu A(R_\eta\|D\vec u-\vec F\|)}+1\right)A(R_\eta\|D\vec u-\vec F\|)\circ \rhotilde\\
&\leq\frac{\mu}{2}\left(\frac{\mu A(R_\eta\|\vec F\|)+4}{\mu A(R_\eta\varepsilon)}+1\right)A(R_\eta\|D\vec u-\vec F\|)\circ\rhotilde\\
&\leq \frac{\mu}{2}c_{\eta,\mu}\left(\frac{\mu c_{\eta,\mu}A(\|\vec F\|)+4}{\mu A(R_\eta\varepsilon)}+1\right)A(\|D\vec u-\vec F\|)\circ\rhotilde
\end{split}
\label{stim1}
\end{equation}
and
\begin{equation}
h(\det D\vec u)\circ \rhotilde\leq |h(\det D\vec u)\circ \rhotilde-h(\det \vec F)|+\frac{h(\det \vec F)}{A(\varepsilon)} \, A(\|D\vec u-\vec F\|)\circ \rhotilde.
\label{stim2}
\end{equation}
Therefore, setting
\begin{equation*}
c':=c_W\cdot \max\left \{1\, , \frac{\mu}{2}c_{\eta,\mu}\left(\frac{\mu c_{\eta,\mu}A(\|\vec F\|)+4}{\mu A(R_\eta\varepsilon)}+1\right)+\frac{h(\det \vec F)}{A(\varepsilon)}\right\},
\end{equation*}
from \eqref{eq: approxboundWdz}, \eqref{stim1}-\eqref{stim2} and \eqref{eq: approxqcledelta} we obtain
\begin{align*}
\int_{U'}W(D\vec z,\vn\circ \vec z)&\leq c_W\int_{U'}(1+h(\det D\rhotilde(\vec x)))\left(2+ A(R_\eta\|D\vec u\|)+h(\det D\vec u)\right)\circ \rhotilde(\vec x)\, \dd \vec x\\
&\leq c'\int_{U'}(1+h(\det D\rhotilde(\vec x)))\left(A(\|D\vec u-\vec F\|)+|h(\det D\vec u)-h(\det \vec F)|\right)\circ \rhotilde(\vec x)\, \dd \vec x\\
&\leq c_\eta' \, \delta \, |B'|,
\end{align*}
where $c_\eta':=c'c_\eta$ depends on $W$, $\eta$, $A$ and $\vec F$ but not on $\delta$.

 In $B'\setminus (U_\gamma\cup U''_\epsilon)$ we have $\|D\vec u-\vec F\|\circ \rhotilde\leq \varepsilon  \leq 1$ and $\det D\hvpe < \gamma$.
  Then we deduce that $\|D\vec u\|\circ \rhotilde\leq \|\vec F\|+1$ and, by virtue of the continuity estimate \eqref{eq: approxconvqcWqccont}, we also obtain $h(\det D\vec u)\circ \rhotilde\leq h(\det \vec F)+1$.
  Now, from \eqref{eq: approxboundWdz} again, we infer that
\begin{align*}
W(D\vec z,\vn\circ \vec z)&\le c\left(3+A(R_\eta(1+\|\vec F\|))+h(\det \vec F)\right)(1+h(\det D\rhotilde))\\
& \leq c\left(3+\frac{\mu}{2}c_{\eta,\mu}\right)\left(A(1)+A(\|\vec F\|)+h(\det \vec F)\right)(1+h (\det D\rhotilde)),
\end{align*}
with $c$ depending only on $W$. Finally, by virtue of \eqref{eq: approxqcdet<gam1} we get
\[\int_{B'\setminus (U''_\epsilon\cup U_\gamma)}W(D\vec z,\vn\circ \vec z)\, \dd \vec x\leq c \, \eta \, |B'|,\]
and, consequently,
\begin{equation}
J_1=\int_{B'}\left(W(D\vec z,\vn\circ \vec z)-W(D\hvpe,\vn\circ \vec z)\right) \dd \vec x\leq (\eta+c_\eta'\delta+c\eta) \, |B'|.
\label{stima1}
\end{equation}

Adding term by term the estimates \eqref{stima1}, \eqref{stima2}, \eqref{stima3} and \eqref{stima4} we find
\begin{align*}
\int_{B'}&\left(W(D\vec z,\vn\circ \vec z)-W^{qc}(D\vec u,\vn\circ \vec u)\right) \dd \vec x\\
&\leq [(5+c)\eta + (c_2+c_4+c'_\eta)\delta]|B'|.
\end{align*}
Since all the constants in the estimate above 
do not depend on $\delta$, we may choose any $\delta$ complying with
\begin{equation*}
\delta<\frac{\eta}{c_2+c_4+c'_\eta}
\end{equation*}
to deduce \eqref{eq: lemconv Wqc} from \eqref{eq: approxledelta}.

With the growth condition \eqref{condW}, we find that $D\vec z\in L^{A}(B)$ and, consequently, $\vec z\in W^{1,A}(B)$. Furthermore, since $\vec a_0$ was chosen so that $\det D\vec z\in L^{1}(B)$, Lemma~\ref{lem3.10} gives $\vec z\in \A_{{\beta}}(B)$ and the proof is concluded.

\endproof

With the next proposition, we obtain the limsup inequality \eqref{eqn:limsup}. The strategy of the proof, based on the argument of \cite[Lemma~6.3]{MCOl} (see also \cite[Lemma~3.3]{CoDo15} for the case with only the mechanical term), is to apply Lemma~\ref{approxubyz} around each Lebesgue point of $D\vec u$ and $\vn\circ\vec u$. We exhibit the construction of a mutual recovery sequence $\{(\vec u_j, \vn_j)\}$ providing a limsup inequality for both the mechanical term \eqref{dis1} and the nematic term \eqref{dis2} separately. The desired inequality \eqref{eqn:limsup} then will follow immediately from the subadditivity of the limsup.

\begin{prop}\label{prop:limsup}
Let $\Omega\subset\R^{n}$ be open, bounded and with Lipschitz boundary, and assume that $W$ satisfies $(W_1)$-$(W_3)$. Then for any $\vec u\in\A_{{\beta}}(\Omega)$ and any $\vn \in W^{1,A}(\imT(\vec u,\Omega),\mS^{n-1})$, there exist two sequences $\vec u_j\in\A_{\textcolor{red}{\beta}}(\Omega)$ and $\vn_j \in W^{1,A}(\imT(\vec u,\Omega),\mS^{n-1})$ such that $\vec u_j\rightharpoonup \vec u$ in $W^{1,A} (\Omega, \R^n)$, $\vec u_j = \vec u$ on $\partial \Omega$, $\imT(\vec u_j,\Omega) = \imT(\vec u,\Omega)$ for all $j \in \N$, $\vn_j \weakc \vn$ in $W^{1,A} (\imT(\vec u,\Omega), \mS^{n-1})$,
\begin{equation}\limsup_{j\to\infty}\int_{\Omega}W(D\vec u_{j},\vn_j\circ \vec u_j)\, \dd \vec x\leq \int_{\Omega}W^{qc}(D\vec u,\vn \circ \vec u)\, \dd \vec x,
\label{dis1}
\end{equation}
and 
\begin{equation}
\limsup_{j\to\infty}\int_{\imT(\vec u_j,\Omega)}V(\vn_j(\vec y), D\vn_j(\vec y)) \, \dd \vec y\leq \int_{\imT(\vec u,\Omega)} V^{tqc}(\vn(\vec y),D\vn(\vec y)) \, \dd \vec y. 
\label{dis2}
\end{equation}
\end{prop}
\proof

The proof of \eqref{dis2} is immediate. Indeed, by virtue of Theorem~\ref{th:tqc} there exists a sequence $\{\vn_k\}_{k \in \N}\subset W^{1,A}(\imT(\vec u,\Omega),\mS^{n-1})$ such that $\vn_k\weakly\vn$ in $W^{1,A}(\imT(\vec u,\Omega),\mS^{n-1})$ and 
\[\lim_{k\to\infty} \int_{\imT(\vec u,\Omega)}V(\vn_k(\vec y), D\vn_k(\vec y)) \, \dd \vec y =  \int_{\imT(\vec u,\Omega)} V^{tqc}(\vn(\vec y),D\vn(\vec y)) \, \dd \vec y. \]
\\
As for \eqref{dis1}, we note that if $\int_{\Omega}W^{qc}(D\vec u,\vn\circ \vec u)\, \dd \vec x=\infty$, we can choose the constant sequence $\vec u_j=\vec u$. Thus, from now on we will assume that $W^{qc}(D\vec u,\vn\circ \vec u)$ is integrable on $\Omega$. We preliminarly prove the following Claim.\\

{\bf Claim:} Let $\{\eta_j\}_{j\in\N}\subset (0,1)$ be a nonincreasing sequence of numbers such that $\eta_j\searrow0$ as $j\to+\infty$. Assume that, for every $j\in\N$, there exists $\vec u_j\in\A_{{\beta}}(\Omega)$ such that 
\begin{equation}
\|\vec u_j - \vec u\|_{L^A(\Omega)}\leq \eta_j, 
\label{near}
\end{equation}
$\vec u_j = \vec u$ on $\partial\Omega$, $\imT(\vec u_j,\Omega)=\imT(\vec u,\Omega)$ and 
\begin{equation}\label{eq:convujiwlewqcseq}
\int_{\Omega}W(D\vec u_j,\vn\circ \vec u_j)\, \dd \vec x\leq \int_{\Omega}W^{qc}(D\vec u,\vn\circ \vec u)\, \dd \vec x+\eta_j.
\end{equation}
Let $\{\vec n_k\}$ be the sequence defined above. Then, up to a subsequence, \eqref{dis1} holds.\\
\emph{Proof of Claim:} First, \eqref{near} implies that $\vec u_j\to \vec u$ in $L^A(\Omega)$ and, with \eqref{eq:convujiwlewqcseq} and \eqref{condW}, we can deduce the uniform bound $\sup_{j \in \N} \|\vec u_j\|_{W^{1,A}} < \infty$. Thus, up to a subsequence (not relabeled), $\vec u_j\weakly \vec u$ in $W^{1,A} (\Omega, \R^n)$. 
On the other hand, since $\vn_k\to \vn$ a.e.\ in $\imT(\vec u,\Omega)$ and $\vec u_j$ satisfies Lusin's $N^{-1}$ condition of Definition~\ref{Ncond} (this is a consequence of the fact that $\det D \vec u_j > 0$ a.e. in $\Omega$; see, e.g., \cite[Lemma~2.8(c)]{HBMC}), for every $j\in\N$ we have $\vn_k\circ \vec u_j\to \vn\circ \vec u_j$ a.e.\ in $\Omega$ as $k\to\infty$. Therefore, using ($W_3$) we obtain
\begin{equation}
\int_{\Omega}W(D\vec u_j,\vn_k\circ \vec u_j)\, \dd \vec x\le \int_{\Omega} \left( \omega(\|\vn_k\circ \vec u_j-\vn\circ \vec u_j\|)+1 \right) W(D\vec u_j,\vn\circ \vec u_j)\, \dd \vec x.
\label{boundd}
\end{equation}
Now, an application of the Theorem of dominated convergence to \eqref{boundd} gives
\[
 \limsup_{k \to \infty} \int_{\Omega} W(D\vec u_j,\vn_k\circ \vec u_j)\, \dd \vec x\leq \int_{\Omega} W(D\vec u_j,\vn\circ \vec u_j)\, \dd \vec x ,
\]
so that for each $j\in\N$ we can take $k_j \in \N$ large enough to have, with \eqref{eq:convujiwlewqcseq}, 
\[\int_{\Omega}W(D\vec u_j,\vn_{k_{j}}\circ \vec u_j)\, \dd \vec x\leq \int_{\Omega}W^{qc}(D\vec u,\vn\circ \vec u)\, \dd \vec x+2\eta_j.\]
Therefore, relabelling the sequence $\{\vn_j\}_{j\in\N}$ we have
\[\limsup_{j\to\infty}\int_{\Omega}W(D\vec u_{j},\vn_j\circ \vec u_j)\, \dd \vec x\leq \int_{\Omega}W^{qc}(D\vec u,\vn \circ \vec u)\, \dd \vec x,\]
and the proof of Claim is complete.\\

From \eqref{condW}, the convexity of $A$ and $h$ and the definition of $W^{qc}$ we have
\[\frac{1}{c_W}A(\|D\vec u\|)+\frac{1}{c_W}h(\det D\vec u)-c_W\leq W^{qc}(D\vec u,\vec m)\text{ for all }
\vec m\in \mS^{n-1} ,\]
since the left-hand side of the inequality above is polyconvex and then quasiconvex. With $W^{qc}(D\vec u,\vn\circ \vec u)\in L^1(\Omega)$, we deduce, in particular, that $A(\|D\vec u\|)$ and $h(\det D\vec u)$ are integrable. On the other hand, we have $\vn\circ \vec u\in L^{\infty}(\Omega, \mS^{n-1})$, since $\vn \circ \vec u$ is measurable in view of \cite[Lemma~7.7]{HBMC}. 

Now we focus on the construction of the sequence $\{\vec u_j\}$, whose existence was assumed in the statement of the Claim.
Denoting by $E$ the intersection of the sets of Lebesgue points of $D\vec u$, $\vn\circ \vec u$ and $h(\det D\vec u)$,  we apply Lemma~\ref{approxubyz} to every point in $E$. For this, given $\vec x\in E$, we fix $\vec F_{\vec x}:=D\vec u(\vec x)$ and $\vec m_{\vec x}:=\vn\circ \vec u(\vec x)$,  and choose $\delta_{\vec x}$ as in Lemma~\ref{approxubyz} for these $\vec F_{\vec x}$, $\vec m_{\vec x}$ and $\eta\in(0,1)$.

Setting $\vec u_0:=\vec u$ and $\Omega_0:=\Omega$, we will construct by induction a sequence $\{(\vec u_j,\Omega_j)\}_{j\in\N}$ such that, for every $j\in\N$,\\
(i) $\vec u_j\in \A_{{\beta}}(\Omega)$;\\
(ii) $\Omega_j\subset\Omega$, $\Omega_{j}\subset\Omega_{j-1}$;\\ 
(iii) $\vec u_j=\vec u$ on $\Omega_j$;\\
(iv) $\imT(\vec u_j,\Omega)=\imT(\vec u,\Omega)$.\\
Assume that the sequence $(\vec u_j,\Omega_j)$ has been constructed until some $j\geq1$. Then $(\vec u_{j+1},\Omega_{j+1})$ is defined from $(\vec u_j,\Omega_j)$ as follows.

   For all $\vec x\in E\cap\Omega_j$ we choose $r_j(\vec x)\in (0,\frac{1}{C_1}\eta^j)$ such that $B(\vec x, r_j(\vec x))\subset\Omega_j$, $\vec u^* \in W^{1,A} (\partial B(\vec x, r_j(\vec x)), \R^n)$ defined by Proposition~\ref{pr:fine}(ii) and 
\[\fint_{B(\vec x,r)}\left(A(\|D\vec u_j(\vec x')-\vec F_{\vec x}\|)+|h(\det D\vec u_j(\vec x'))-h(\det \vec F_{\vec x})|+A(\|\vn\circ \vec u_j(\vec x')-\vec m_{\vec x}\|)\right) \dd \vec x'\leq \delta_{\vec x}\]
for all $r<r_j(\vec x)$. The union of this collection of balls $B(\vec x,r_j(\vec x))$ covers $\Omega_j$ up to a null set. From this covering, we extract a finite disjoint family $\{B(\vec x_k,r_k)\}_{k=0}^{M}$ such that 
\[\left|\bigcup_{k=0}^{M}B(\vec x_k,r_k) \right|\ge \frac{1}{2}|\Omega_j|.\]
We define $\vec u_{j+1}=\vec u_j$ on $\Omega\setminus \bigcup_{k=0}^{M}B(\vec x_k,r_k)$ and as the function $\vec z$ of Lemma \ref{approxubyz} in each of the balls $B(\vec x_k,r_k)$. Then $\vec u_{j+1}=\vec u_j=\vec u$ on $\partial\Omega $ and by virtue of Lemma~\ref{lem3.10}, we get $\vec u_{j+1}\in \A_{{\beta}}(\Omega)$.
Now, let $B(\vec x_k',\frac{r_k}{2} )\subset B(\vec x_k,r_k)$ be the ball given by Lemma \ref{approxubyz} and choose an increasing sequence $\{U_i\}_{i\in\N}$ of open sets such that $U_i\ssubset\Omega$, $\bigcup_{i\in\N} U_{i}=\Omega$, $\bigcup_{k=0}^{M}B(\vec x_k,r_k) \subset U_1$ and $\vec u^*_j \in W^{1,A} (\partial U_i, \R^n)$ for all $i \in \N$.
Then, $\vec u_j$ and $\vec u_{j+1}$ coincide in a neighbourhood of each $\partial U_i$, so that $\vec u^*_{j+1} \in W^{1,A} (\partial U_i, \R^n)$ and $\imT(\vec u_j,U_i) = \imT(\vec u_{j+1},U_i)$, since the degree only depends on the boundary values.

Therefore, $\imT(\vec u_j,\Omega) = \imT(\vec u_{j+1},\Omega)$, and, by applying (iv) iteratively, $\imT(\vec u_{j+1},\Omega)=\imT(\vec u,\Omega)$.
 
Moreover, by Lemma \ref{approxubyz} applied with $\frac{1}{C_1}\eta^{j+1}$ in place of $\eta$, we obtain
 \begin{equation}\label{eq:sucesuj WleWqc}
 \int_{B\left(\vec x_k',\frac{r_k}{2}\right)}W(D\vec u_{j+1},\vn\circ \vec u_{j+1})\, \dd \vec x\le \int_{B\left(\vec x_k',\frac{r_k}{2}\right)}\left(W^{qc}(D\vec u,\vn\circ \vec u)+\eta^{j+1}\right) \dd \vec x
 \end{equation}
 and 
  \begin{equation}\label{eq:sucesuj LpleWqc}
 \int_{B\left(\vec x_k,r_k\right)}A(\|\vec u_{j+1}-\vec u\|)\, \dd \vec x\leq C\int_{B(\vec x_k,r_k)}\left(W^{qc}(D\vec u,\vn\circ \vec u)+1\right) \dd \vec x.
 \end{equation}
We set $\Omega_{j+1}=\Omega_{j}\setminus \bigcup_{k=0}^{M}\overline{{B}\left(\vec x_k',\frac{r_k}{2}\right)}$.
Clearly, $\vec u_{j+1}=\vec u_j=\vec u$ on $\Omega_{j+1}$, $\Omega_{j+1}\subset\Omega_j$ and it holds that $|\Omega_{j+1}|\le (1-2^{-n-1})|\Omega_j|$. In particular,
\begin{equation}
|\Omega_j|\le (1-2^{-n-1})^{j}|\Omega|\to 0.
\label{eqn:measzero}
\end{equation}
Thus, $(\vec u_{j+1},\Omega_{j+1})$ complies with (i)-(iv) above and the construction is complete.

Now, we are left to show that for $j$ large enough, $\vec u_j$ satisfies \eqref{eq:convujiwlewqcseq} and that $\|\vec u_j-\vec u\|_{L^{A}(\Omega)}$ is uniformly small.
As for the latter, from \eqref{eq:sucesuj LpleWqc} we have 
\[\int_{\Omega}A(\|\vec u_{j}-\vec u\|)\, \dd \vec x\leq C \int_{\Omega}\left(W^{qc}(D\vec u,\vn\circ \vec u)+1\right) \dd \vec x , \]
so $\vec u_j$ is close to $\vec u$ in $L^A(\Omega)$, independently of $j$. 
On the other hand, from \eqref{eq:sucesuj WleWqc} we obtain
\[ \int_{\Omega\setminus\Omega_j}W(D\vec u_{j+1},\vn\circ \vec u_{j+1})\, \dd \vec x\leq \int_{\Omega\setminus\Omega_j}\left(W^{qc}(D\vec u,\vn\circ \vec u)+\eta^{j+1}\right) \dd \vec x,\]
which implies
 \[ \int_{\Omega}W(D\vec u_{j+1},\vn\circ \vec u_{j+1})\, \dd \vec x\leq \int_{\Omega\setminus \Omega_j}\left(W^{qc}(D\vec u,\vn\circ \vec u)+\eta^{j+1}\right) \dd \vec x+\int_{\Omega_j}W(D\vec u,\vn\circ \vec u)\, \dd \vec x.\]
 Using \eqref{eqn:measzero} and the fact that, from \eqref{condW}, $W(D\vec u,\vn\circ \vec u)$ is integrable (since $A(\|D\vec u\|)$ and $h(\det D\vec u)$ are integrable), for $j$ large enough we get 
 \begin{equation*}
 \int_{\Omega}W(D\vec u_{j+1},\vn\circ \vec u_{j+1})\, \dd \vec x\leq \int_{\Omega}\left(W^{qc}(D\vec u,\vn\circ \vec u)+2\eta^{j+1}\right) \dd \vec x
 \end{equation*}
and the proof is concluded.

\endproof

\subsection{Relaxation}

The following general abstract result (see, e.g., \cite[Proposition~11.1.1, Theorem~11.1.1]{AtBuMi06}) will allow us to identify $I^*$ with the lower semicontinuous envelope of the energy $I$ with respect to the $L^1 (\Omega, \R) \times L^1 (\R^n, \R^n)$ topology.

\begin{prop}\label{prop11.1.1}
The function defined, for every $(\vec u,\vn) \in L^1 (\Omega, \R) \times L^1 (\R^n, \R^n)$, as
\begin{equation}
 \displaystyle \bar{I}(\vec u,\vn)= \inf \left\{ \liminf_{j\to\infty} I(\vec u_j,\vn_j) : (\vec u_j, \vn_j) \to (\vec u, \vn) \text{ as } j \to \infty \text{ in } L^1 (\Omega, \R) \times L^1 (\R^n, \R^n) \right\}
\label{envelop}
\end{equation}
is the lower semicontinuous envelope of $I$ with respect to the $L^1 (\Omega, \R) \times L^1 (\R^n, \R^n)$ topology; i.e., the greatest lower semicontinuous function less than $I$. Moreover, the function $\bar{I}$ is characterized by the following assertions:
\\
\begin{description}
\item[(i)]  for every $(\vec u_j, \vn_j) \to (\vec u, \vn)$,\quad $\bar{I} (\vec u, \vn) \leq \displaystyle\liminf_{j \to \infty} I (\vec u_j, \vn_j)$;\\
\item[(ii)] there exists a sequence $(\bar{\vec u}_j, \bar{\vn}_j) \to (\vec u, \vn)$ such that $\displaystyle\limsup_{j\to\infty} I (\bar{\vec u}_j, \bar{\vn}_j) \leq \bar{I}(\vec u, \vn)$.
\end{description}
\end{prop}

Since, by definition, $I^*(\vec u, \vn)\leq I(\vec u, \vn)$,  Proposition~\ref{prop11.1.1} with Proposition~\ref{prop:lowersemicont}(ii) and Theorem~\ref{recseq} imply that $I^*(\vec u, \vn)=\bar{I}(\vec u, \vn)$. This result, combined with the compactness theorem (Proposition~\ref{prop:lowersemicont}(i)) and the general relaxation theorem in countable topological spaces (see, e.g., \cite[Theorem~11.1.2]{AtBuMi06}) gives the following final relaxation theorem.

\begin{thm}\label{relaxation}
Let $W$ satisfy $(W_1)$-$(W_3)$ and let $V$ satisfy $(V_1)$. Assume $W^{qc}$ is polyconvex.
Then $I^*$ is the lower semicontinuous envelope of $I$ with respect to the $L^1 (\Omega, \R) \times L^1 (\R^n, \R^n)$ topology and coincides with \eqref{envelop}. If, in addition, $I\not\equiv +\infty$, then 
\\
\begin{description}
\item[(a)] $I^*$ admits a minimizer;

\item[(b)] For every minimizer $(\bar{\vec u}, \bar{\vn})$ of $I^*$, there exists $(\vec u_j,\vn_j)$ a minimizing sequence for $I$ such that $(\vec u_j, \vn_j) \to (\bar{\vec u}, \bar{\vn})$ in $L^1 (\Omega, \R) \times L^1 (\R^n, \R^n)$;

\item[(c)] Every minimizing sequence $(\vec u_j,\vn_j)$ of $I$ converges, up to a subsequence in $L^1 (\Omega, \R) \times L^1 (\R^n, \R^n)$, to a minimizer $(\bar{\vec u}, \bar{\vn})$ of $I^*$.
\end{description}
\end{thm}

\section*{Acknowledgements}

{The authors are members of Gruppo Nazionale per l’Analisi Matematica, la Probabilità e le loro Applicazioni (GNAMPA) of INdAM.} The research of B. S.\ was supported by University of Naples Project VAriational TECHniques in Advanced MATErials (VATEXMATE) {and by PRIN Project 2017TEXA3H.} The project was carried out during the visit of B. S. to the Isaac Newton Institute for Mathematical Sciences. She would like to thank the Isaac Newton Institute for Mathematical Sciences for support and hospitality during the programme \lq\lq ~The mathematical design of new materials \rq \rq  when work on this paper was undertaken. This work was supported by: EPSRC grant number EP/R014604/1". G. S. was supported by the Italian Ministry of Education, University and Research through the Project “Variational methods for stationary and evolution problems with singularities and interfaces” (PRIN 2017). The authors gratefully acknowledge the anonymous referee for a careful reading of the paper and for the interesting remarks leading to improvements of the manuscript.

\Addresses

\end{document}